\documentclass{amsart}
\allowdisplaybreaks[4]
\usepackage{graphicx}
\usepackage{color}

\newcommand{\Z}{\mathbf{Z}}
\newcommand{\R}{\mathbf{R}}

\newcommand{{\ba}}{\bf a}
\newcommand{\ve}{\varepsilon}
\newcommand{\la}{\lambda}
\newcommand{\La}{\Lambda}
\newcommand{\ga}{\gamma}
\newcommand{\Ga}{\Gamma}
\newcommand{\pa}{\partial}
\newcommand{\ra}{\rightarrow}

\newcommand{\del}{\delta}

\newcommand{\al}{\alpha}
\newcommand{\e}{\mathrm{e}}

\newcommand{\be}{\begin{equation}}
\newcommand{\ee}{\end{equation}}

\newtheorem{lem}{Lemma}{\bf}{\it}
\newtheorem{rem}{Remark}{\it}{\rm}
{\it}{\rm}

\newtheorem{theorem}{Theorem}
\newtheorem{proposition}{Proposition}
\newtheorem{corollary}{Corollary}
\numberwithin{theorem}{section}
\numberwithin{lem}{section}
\numberwithin{equation}{section}
\numberwithin{proposition}{section}
\numberwithin{corollary}{section}
\title[Global Stability]{On Global Stability for Lifschitz-Slyozov-Wagner like  equations }

\author{Joseph G. Conlon and Barbara Niethammer}

\address{ (Joseph G. Conlon): University of Michigan\\ Department of Mathematics\\ Ann Arbor,
  MI 48109-1109}
\email{conlon@umich.edu}

\address{ (Barbara Niethammer): Oxford Centre for Nonlinear PDE \\ Mathematical Institute,  University of Oxford \\ 24-29 St. Giles', Oxford, OX1 3LB, United Kingdom }
 \email{  niethammer@maths.ox.ac.uk } 


\begin{document}

\maketitle

\begin{abstract}
This paper is concerned with  the stability and asymptotic stability  at large time of solutions to a system of equations, which includes the Lifschitz-Slyozov-Wagner (LSW) system in the case when the initial data has compact support. The main result of the paper is a proof of weak global asymptotic stability for LSW like systems.  Previously strong local asymptotic stability results were obtained by Niethammer and  Vel\'{a}zquez for the LSW system with initial data of compact support.  Comparison to a quadratic model plays an important part in the proof of the main theorem when the initial data is  critical. The quadratic model  extends the linear model of Carr and Penrose, and has a time invariant solution which decays exponentially at the edge of its support in the same way as the infinitely differentiable self-similar solution of the LSW model. 
\end{abstract}

\section{Introduction.}
In this paper we continue the study of the large time behavior of solutions to the  Lifschitz-Slyozov-Wagner (LSW) equations \cite{ls,w} begun in \cite{c}. The LSW equations occur in a variety of contexts \cite{pego, penrose} as a mean field approximation for the evolution of particle clusters of various volumes. Clusters of volume $x>0$ have density $c(x,t)\ge 0$ at time $t>0$. The density evolves according to a linear law, subject to the linear mass conservation constraint as follows:
\begin{eqnarray}
\frac{\pa c(x,t)}{\pa t}&=& \frac{\pa}{\pa x}\Big( \Big( 1- \big( x L^{-1}(t)\big)^{1/3} \Big) c(x,t)\Big),  \quad x>0, \label{P1}\\
\int_0^\infty x c(x,t) dx&=& 1. \label{Q1} 
\end{eqnarray}
One wishes then to solve (\ref{P1}) for $t>0$ and initial condition $c(x,0) = c_0(x) \ge 0, \ x > 0$, subject to the constraint (\ref{Q1}).  The parameter $L(t) > 0$ in (\ref{P1}) is determined by the constraint (\ref{Q1}) and is therefore given by the formula,
\be \label{R1}
L(t)^{1/3} = \int^\infty_0 \ x^{1/3} c(x,t)dx \Big/ \int^\infty_0 c(x,t) dx.
\ee
Evidently then $L(t)^{1/3}$ is  the average cluster radius at time $t$ and the time evolution of the LSW system is in fact non-linear.  Existence and uniqueness of solutions to (\ref{P1}), (\ref{Q1}) with given initial data $c_0(x)$ satisfying the constraint has been proven  in \cite{laurencot} (see also  \cite{cg}) for integrable functions $c_0(\cdot)$, and in \cite{np2} for initial data such that $c_0(x)dx$ is an arbitrary Borel probability measure with compact support. In \cite{np3}   the methods of \cite{np2} are further developed to prove existence and uniqueness for initial data such that $c_0(x)dx$ is a Borel probability measure with finite first moment.

The main focus of \cite{c} and the current paper is to understand the phenomenon of {\it coarsening} for the LSW system. Specifically, beginning with rather arbitrary initial data satisfying the constraint (\ref{Q1}), one expects the typical cluster volume to increase linearly in time.  This is a consequence of the dilation invariance of the system. That is if the function $c(x,t), \ x,t>0$, is a solution of (\ref{P1}), (\ref{Q1}), then for any parameter $\lambda>0$ so also is the function $\lambda^2 c(\lambda x,\lambda t)$. Letting $\Lambda(t)$ be the mean cluster volume at time $t$,
\be \label{S1}
\Lambda(t)=\int_0^\infty xc(x,t)dx \Big/\int_0^\infty c(x,t) dx, \quad t\ge 0,
\ee
one expects  $\Lambda(t)\sim C t$ at large $t$ for some constant $C>0$. The problem of proving that typical cluster volume  increases linearly in time is subtle since it is easy to see that the constant $C$ depends on detailed properties of the initial data. In fact if the initial data is a Dirac delta measure then  $C=0$. Less trivially one can construct a  family of self-similar solutions \cite{np1} to (\ref{P1}), (\ref{Q1}) depending on a parameter  $\beta$, which may take any value in the interval $0<\beta\le 1$. In that case $\Lambda(t)\sim C(\beta) t$ at large $t$, where $0<C(\beta)<\beta$. The main result of \cite{c} is  an upper and lower bound on the rate of coarsening of the LSW model  for a large class of initial data: there exist positive constants $C_1,C_2$ depending only on the initial data such that
\be \label{T1}
C_1T \ \le \La(T) \ \le C_2 T \quad {\rm for \ } T\ge 1.
\ee
The class of initial data for which (\ref{T1}) holds includes the exponential function $c_0(x)=e^{-x}, \ 0\le x<\infty,$ and the slowly decreasing functions   $c_0(x)=K_\ve/(1+x)^{2+\ve}, \ 0\le x<\infty,$ where we  require $\ve>0$ in order to satisfy the conservation law (\ref{Q1}).  It also includes initial data with compact support such as  $c_0(x)=K_p(1-x)^{p-1},\  0\le x\le 1, \ \ c_0(x)=0, \ x>1$, where here we require  $p>0$ so that (\ref{Q1}) holds.    A time averaged upper bound  on the rate of coarsening for such a wide class of initial data was already known from a result of  Dai and Pego \cite{dp}, which  applies  the Kohn-Otto argument \cite{ko}  to the LSW system. 

In this paper we shall be confining our investigation of the LSW system to solutions of (\ref{P1}), (\ref{Q1})  which have  initial data with compact support.  It is easy to see that if the initial data  $c_0(\cdot)$ for (\ref{P1}) has compact support then the solution $c(\cdot,t)$ at any later time $t>0$ also has compact support.  Furthermore all self-similar solutions of (\ref{P1}), (\ref{Q1}) have compact support.  The study of solutions to (\ref{P1}), (\ref{Q1})  with initial data which has compact support generally proceeds \cite{np1} by  normalizing the support of the function $c(\cdot,t)$ to  be the interval $0\le x\le 1$ for all $t\ge 0$.  Denoting this normalized density also by $c(\cdot,t)$, we define functions  $w(\cdot,t)\ge 0, \ h(\cdot,t)\ge 0$  by the formulas
\be \label{N1}
w(x,t) \ = \ \int_x^1 c(x',t) \ dx' \ , \quad h(x,t) \ = \ \int_x^1 w(x',t) \ dx'  \ , \quad 0\le x <1.
\ee
Then  the dynamical evolution of solutions to the LSW system  is governed by the PDE 
\begin{equation} \label{A1}
\frac{\pa w(x,t)}{\pa t}+ \left[\phi(x)-\kappa(t)\psi(x)\right] \frac{\pa w(x,t)}{\pa x}=w(x,t), \quad 0\le x<1, \ t\ge 0,
\end{equation}
with the mass conservation law
\begin{equation} \label{B1}
h(0,t) \ = \ \int_{0}^1  w(x,t)  \ dx=1 \  ,  \quad t\ge 0.
\end{equation} 
where the functions $\phi(\cdot)$ and $\psi(\cdot)$ in (\ref{A1}) are given by the formulas,
\begin{equation} \label{C1}
\phi(x)=x^{1/3}-x, \quad \psi(x)=1-x^{1/3}, \quad 0\le x \le 1.
\end{equation}
The initial data $w_0(\cdot)$  for (\ref{A1}), (\ref{B1}) is now taken to be a non-negative decreasing strictly positive function $w_0(x), \ 0\le x<1$, which converges to $0$ as $x\ra 1$.  This implies that the solution $w(x,t)$ of  (\ref{A1}), (\ref{B1})  also is non-negative decreasing strictly positive  in $x$ for $0\le x<1$ and converges to $0$ as $x\ra 1$. The function $\kappa(\cdot)$ in (\ref{A1}) is uniquely determined by the conservation law (\ref{B1}) just as $L(\cdot)$ in (\ref{P1}) is determined from (\ref{Q1}). 

The inequality (\ref{T1}) was proven in \cite{c} by making use of the properties of a certain function of the solution of (\ref{P1}) which we called the {\it beta function}. The beta function $\beta(\cdot,t)$ associated with the solution $w(\cdot,t)$ of (\ref{A1}) is  given by the formula
\be \label{M1}
\beta(x,t) \ = \ \frac{c(x,t) h(x,t)}{w(x,t)^2} \ , \quad 0\le x<1,
\ee
where  $c(\cdot,t)$ and $w(\cdot,t)$ are as in (\ref{N1}). It was shown in \cite{c} that if the beta function of the initial data for (\ref{A1}), (\ref{B1}) satisfies
\be \label{U1}
\lim_{x\ra 1} \beta(x,0) \ = \ \beta_0>0 \ ,
\ee
then the coarsening inequality (\ref{T1}) holds. Since the support of the function $w_0(\cdot)$ is the interval $0\le x\le 1,$ it is easy to see that if (\ref{U1}) holds then one must have $\beta_0\le 1$. 
We shall refer to initial data $w_0(\cdot)$ for (\ref{A1}), (\ref{B1}) as being {\it subcritical} if  (\ref{U1}) holds with $0<\beta_0<1$, and {\it critical} if (\ref{U1}) holds with $\beta_0=1$.  Examples of subcritical and critical initial data are given by functions $w_0(\cdot)$,
\be \label{V1}
w_0(x)=(1-x)^p \  \ {\rm for \ } p>0, \quad w_0(x)=\exp\left[-\frac{1}{1-x}\right] \ , \quad 0\le x<1.
\ee
In (\ref{V1}) the first function   has $\beta_0=p/(1+p)<1,$ and  the second function $\beta_0=1$. 

Self-similar solutions of (\ref{P1}), (\ref{Q1}) correspond to time independent solutions of (\ref{A1}), (\ref{B1}). There is an infinite family of such time independent solutions  characterized by a parameter $\kappa\ge \kappa_0=\phi'(1)/\psi'(1)>0$.  These solutions $w_\kappa(x)$  can be easily distinguished by their behavior as $x\ra 1$ as follows: 
\begin{eqnarray} \label{G1}
{\rm for \ }  \kappa > \kappa_0,  \ &w_\kappa(x) &\sim (1-x)^{p}, \qquad   \qquad 1/p=(\kappa-\kappa_0)|\psi'(1)|   , \\
{\rm for \ }  \kappa = \kappa_0, \  &w_\kappa(x) &\sim \exp[-1/\ga(1-x)],   \quad  \ga=\kappa_0\psi''(1)-\phi''(1).
\nonumber
\end{eqnarray}
Letting $\beta_\kappa(\cdot)$ denote the beta function (\ref{M1}) corresponding  to $w_\kappa(\cdot)$, it is easy to see that
\be \label{W1}
\kappa \ = \ [ 1/\lim_{x\ra 1} \beta_\kappa(x) -\phi'(1)-1]/|\psi'(1)|  \ ,
\ee
so that $w_\kappa(\cdot)$ is subcritical for $\kappa>\kappa_0$ and critical when $\kappa=\kappa_0$.  It was shown in \cite{nv1} that if the solution $w(\cdot,t)$ of (\ref{A1}), (\ref{B1}) converges as $t\ra\infty$ to $w_\kappa(\cdot)$ with $\kappa>\kappa_0,$ then the initial data $w(\cdot,0)$ must be {\it regularly varying} with exponent $p$  given by (\ref{G1}). Furthermore if the initial data is sufficiently close in the regular variation sense to $w_\kappa(\cdot)$, then $\lim_{t\ra\infty} w(x,t)=w_\kappa(x)$ uniformly on any compact subset of $[0,1)$.  This in turn implies that the average volume (\ref{S1})  satisfies $\lim_{T\ra\infty} \La(T)/T=C>0$. In \cite{c} it was observed that  if the beta function (\ref{M1}) corresponding to the initial data $w(\cdot,0)$ satisfies (\ref{U1}) with $\beta_0=p/(1+p)<1$, then $w(\cdot,0)$ must be {\it regularly varying} with exponent $p$, and that these two conditions are virtually equivalent  (see Lemma 4 of \cite{c} and the remark following). 

The main result of \cite{nv1} can be considered  a {\it strong local} asymptotic stability result  for the LSW model with subcritical initial data.  A corresponding result  for critical initial data was proven in \cite{nv2}. Again it was shown that if the solution $w(\cdot,t)$ of (\ref{A1}), (\ref{B1}) converges as $t\ra\infty$ to $w_{\kappa}(\cdot)$ with $\kappa=\kappa_0,$ then the initial data $w(\cdot,0)$ must satisfy a certain criterion-equation (\ref{A4}) of the present paper. If the initial data is sufficiently close in the sense of this criterion to $w_\kappa(\cdot)$, then $\lim_{t\ra\infty} w(x,t)=w_\kappa(x)$ uniformly on any compact subset of $[0,1)$.   We show in $\S4$ that if (\ref{U1}) holds with $\beta_0=1$ then the criterion of \cite{nv2} for the initial data of (\ref{A1}), (\ref{B1}) is satisfied.

Our goal in the present paper is  to prove {\it weak global}  asymptotic stability results  corresponding to the {\it strong local} asymptotic stability results of \cite{nv1,nv2}.
It will be useful to our study  to generalize the system (\ref{A1}), (\ref{B1}), (\ref{C1}) by  allowing more general  functions $\phi(\cdot)$ and $\psi(\cdot)$ on $[0,1]$ than (\ref{C1}). We do however require these functions to be continuous on $[0,1]$, twice continuously differentiable  on $(0,1]$, and have the properties:
\begin{eqnarray}
\quad &\phi(x) \  {\rm is \  concave \  and \  satisfies \ }& \quad \phi(0)=\phi(1)=0, \ -1< \phi'(1)<0.  \label{D1}\\
\quad &\psi(x)  \  {\rm is \  convex \  and \  satisfies \ }&\quad \psi(1)=0,  \ \  \psi'(1)<0, \ \psi''(1)-\phi''(1)>0.  \label{E1}
\end{eqnarray}
 Evidently the conditions (\ref{D1}), (\ref{E1}) imply that the functions $\phi(x), \ \psi(x)$ are strictly positive for $0<x<1$. The conservation law (\ref{B1}), when combined with (\ref{A1}), implies that the parameter
 $\kappa(t)$ is given in terms of $w(\cdot, t)$ by the formula,
\begin{equation} \label{F1} 
\frac{1}{\kappa(t)}\left[ \int_0^1 [1+\phi'(x)]w(x,t)dx\right]= \psi(0)w(0,t)+\int_0^1\psi'(x)w(x,t)dx.
\end{equation}
 One can see from the conditions  (\ref{D1}), (\ref{E1})  and the fact that the function $w(\cdot,t)$ is non-negative decreasing, that $\kappa(t)$ as determined by (\ref{F1}) is positive. Hence the coefficient $\phi(\cdot)-\kappa(t)\psi(\cdot)$  of $\pa w(\cdot,t)/\pa x$ in (\ref{A1}) is concave  for all $t\ge 0$. As in the LSW case there is an infinite family of time independent solutions   of (\ref{A1}) characterized by a parameter $\kappa\ge \kappa_0=\phi'(1)/\psi'(1)>0$ which have the properties (\ref{G1}), (\ref{W1}). 
 
 Our first result is a weak global asymptotic stability result for (\ref{A1}), (\ref{B1}) in the case when  the initial data is subcritical. In order to prove it we need to make a further assumption on the functions $\phi(\cdot), \ \psi(\cdot)$ beyond (\ref{D1}), (\ref{E1}), namely that
 \be \label{O1}
 \phi(\cdot), \ \psi(\cdot)  \ {\rm are \ } C^3 \ {\rm on \ } (0,1]  \ {\rm \ and \ } \phi'''(x)\ge0, \  \psi'''(x)\le 0  \   {\rm for \ } 0<x\le 1.
 \ee
 Evidently (\ref{O1}) holds for the LSW functions (\ref{C1}). 
\begin{theorem} Let $w(x,t), \ x,t\ge 0$, be the solution to (\ref{A1}), (\ref{B1}) with coefficients satisfying (\ref{D1}), (\ref{E1}) and assume that the initial data $w(\cdot,0)$ has beta function $\beta(\cdot,0)$ satisfying (\ref{U1}) with $0<\beta_0<1$.  Then there is a positive constant $C_1$ depending only on the initial data such that $\kappa(t)\ge C_1$ for all $t\ge  0$. If in addition (\ref{O1}) holds, then there is a positive constant $C_2$  depending only on the initial data such that $\kappa(t)\le C_2$ for all $t\ge  0$ and
\be \label{J1}
\lim_{T\ra \infty}\frac{1}{T}\int_0^T \kappa(t) \ dt \ = \ [1/\beta_0-\phi'(1)-1]/|\psi'(1)| \ .
\ee
\end{theorem} 
In the LSW case the condition $C_1\le\kappa(t)\le C_2, \ t\ge 0,$ implies that the ratio of the mean cluster radius to maximum cluster radius is uniformly bounded strictly between $0$ and $1$ for $t\ge 0$.  We prove Theorem 1.1 in $\S2$  by extending the methodology  of the beta function developed in \cite{c}. In order to prove a version of the theorem for critical initial data we have had to have recourse to a different approach. The approach is based on the observation that when the functions $\phi(\cdot), \ \psi(\cdot)$ are quadratic, then the  generally infinite dimensional dynamical system (\ref{A1}), (\ref{B1}) reduces to a two dimensional system. One way of seeing this is to note that for quadratic $\phi(\cdot), \ \psi(\cdot)$ the commutator of the operators $A, \ B$ defined by
\begin{equation} \label{I1}
A=\phi(x)\frac{d}{dx},  \quad B= \psi(x)\frac{d}{dx},
\end{equation}
is a linear combination of $A$ and $B$. Thus $A$ and $B$ generate a two dimensional Lie algebra. The corresponding two dimensional dynamical system can be analyzed in detail and so we are able to prove in $\S3$ and  $\S5$ strong global asymptotic stability for the time independent solutions (\ref{G1}) of (\ref{A1}), (\ref{B1}) .
\begin{theorem} Assume that the functions $\phi(\cdot), \ \psi(\cdot)$ are quadratic, and that  the initial data $w(\cdot,0)$ for (\ref{A1}), (\ref{B1})  has beta function $\beta(\cdot,0)$ satisfying
(\ref{U1}).    Then setting $\kappa=[1/\beta_0-\phi'(1)-1]/|\psi'(1)|$, one has for $\beta_0<1$, 
\be \label{K1}
\lim_{t\ra\infty} \kappa(t ) =   \kappa, \quad \lim_{t\ra\infty} \|\beta(\cdot,t)-\beta_\kappa(\cdot)\|_\infty=0,
\ee
where $\beta_\kappa(\cdot)$ is the beta function of the time independent solution $w_\kappa(\cdot)$ of (\ref{G1}). If $\beta_0=1$ then for any $\ve$ with $0<\ve<1$, one has
\be \label{L1}
\lim_{t\ra\infty} \kappa(t ) =   \kappa_0, \quad \lim_{t\ra\infty} \sup_{0\le x\le 1-\ve}|\beta(x,t)-\beta_{\kappa_0}(x)|=0.
\ee
\end{theorem}
In $\S5$ we note that the convergence result (\ref{L1}) for critical initial data can be improved if we make the further assumption on the initial data:
\be \label{X1}
{\rm There \ exists \ } \del>0 \ {\rm such \ that \ } \beta(x,0)\le 1 \ {\rm for \ } 1-\del\le x<1.
\ee
Thus if (\ref{U1}) with $\beta_0=1$ and (\ref{X1}) hold, then $\lim_{t\ra\infty} \|\beta(\cdot,t)-\beta_{\kappa_0}(\cdot)\|_\infty=0$. The condition (\ref{X1}) turns out to be important for us  when we seek to extend Theorem 1.1 to the case of critical initial data. We also need an extra assumption on the functions $\phi(\cdot), \ \psi(\cdot)$ beyond (\ref{D1}), (\ref{E1}) and (\ref{O1}).  The assumption is as follows:
\begin{multline} \label{Y1}
{\rm The \  function \ } x\ra \phi'(x)+\phi'(1)-\phi(x)[\psi'(x)+\psi'(1)]/\psi(x) \\
 {\rm is \ decreasing \ for \ } \ \ 0\le x<1.
\end{multline}
One can easily see that the LSW functions (\ref{C1}) satisfy (\ref{Y1}). 
\begin{theorem}
Let $w(x,t), \ x,t\ge 0$, be the solution to (\ref{A1}), (\ref{B1}) with coefficients satisfying (\ref{D1}), (\ref{E1}), (\ref{O1}),(\ref{Y1}) and assume that the initial data $w(\cdot,0)$ has beta function $\beta(\cdot,0)$ satisfying (\ref{U1}) with $\beta_0=1$.   If $\lim_{x\ra0}\phi(x)/x=\infty,$ then there exist positive constants $C_1,C_2$ depending only on the initial data such that $C_1\le\kappa(t)\le C_2,$ for $t\ge 0,$ and (\ref{J1}) holds. If the functions $\phi(\cdot),\psi(\cdot)$ are $C^2$ on the closed interval $[0,1]$ and in addition the initial data satisfies (\ref{X1}),  then there exist positive constants $C_1,C_2$ depending only on the initial data such that $C_1\le\kappa(t)\le C_2,$ for $t\ge 0,$ and (\ref{J1}) holds. 
\end{theorem}
Since the LSW function $\phi(\cdot)$ of (\ref{C1}) satisfies $\lim_{x\ra0}\phi(x)/x=\infty,$ Theorem 1.1 implies that weak global asymptotic stability holds for solutions of the LSW system with critical initial data as defined by (\ref{U1}) with $\beta_0=1$. It seems at first surprising that the system (\ref{A1}), (\ref{B1}) is more stable when the function $\phi(\cdot)$ has a singularity at $x=0$. Proposition 4.2 however and the remark following  indicates why this may be the case. The proof of Theorem 1.3 is contained in $\S4$ and $\S6$. In $\S4$ we use the methodology of the beta function to prove certain results, in particular some bounds on the function $\kappa(\cdot)$.    In order to prove the asymptotic stability result (\ref{J1}),  we transform in $\S6$ the system (\ref{A1}), (\ref{B1}) to a system which can be compared to the quadratic model. Hence our proof of asymptotic stability in the  critical case hinges on viewing  (\ref{A1}), (\ref{B1}) as a perturbation of the {\it quadratic model}.  In contrast, the proof of asymptotic stability in the  subcritical case can be accomplished by using the properties of the beta function alone. In \cite{c} it was observed that the methodology of the beta function is a way of viewing the system (\ref{A1}), (\ref{B1}) as a perturbation of the {\it linear model} studied by Carr and Penrose \cite{carr,cp}.  Since there is no critical time independent solution $w_{\kappa_0}(\cdot)$ of (\ref{A1}), (\ref{B1}) for the linear model, it is therefore not surprising that in the proof of asymptotic stability for the critical case one needs to go beyond the methodology of the beta function.

\section{Global Stability for subcritical initial data}
In this section we shall prove Theorem 1.1. First recall that the solution $w(x,t)$ to (\ref{A1}) is given in terms of the initial data $w_0(\cdot)$ by the formula $w(x,t)=e^t w_0(F(x,t)), \ 0\le x \le 1$, where  the mapping $F(\cdot,t)$ is defined by $F(x,t)=x(0)$, with $x(s), \ 0\le s\le t$, being the solution to the terminal value problem
\be \label{A2}
\frac{d x(s)}{ds}= \phi(x(s))-\kappa(s)\psi(x(s)), \quad s\le t, \ x(t)=x.
\ee
The derivative $\pa F(x,t)/\pa x$ is given in terms of the solution to (\ref{A2}) by the formula
\be \label{B2}
\frac{\pa F(x,t)}{\pa x} \ = \  \exp\left[- \int_0^t \{\phi'(x(s))-\kappa(s)\psi'(x(s))\}\ ds \right] \ .
\ee
By virtue of our assumptions (\ref{D1}), (\ref{E1}) and the positivity of the function $\kappa(\cdot)$, it follows from (\ref{B2}) that $F(x,t)$ is a convex function of $x, \ 0\le x \le 1$.  
\begin{lem} 
Let $F(\cdot,\cdot)$ be defined by  (\ref{A2}), where $\kappa(\cdot)$ is determined by the solution of (\ref{A1}), (\ref{B1}). Then $F(0,t)$ is an increasing function of $t$ and $\lim_{t\ra\infty} F(0,t)=1$.
\end{lem}
\begin{proof}
Evidently $F(0,t)$ is an increasing function of $t$, whence $\lim_{t\ra\infty} F(0,t)=\al\le 1$. The conservation law (\ref{B1}) is equivalent to
\be \label{C2}
\int_{F(0,t)}^1 w_0(z)/[\pa F(x,t)/\pa x] \ dz \ = \ e^{-t} \ ,
\ee
where the variables $z$ and $x$ are related by $z=F(x,t)$. From (\ref{D1}), (\ref{E1}) and (\ref{B2}) we see that $\pa F(x,t)/\pa x \le \exp[-t\phi'(1)], \ 0\le x\le 1$, whence (\ref{C2}) implies that
\be \label{D2}
\int_{F(0,t)}^1 w_0(z)\ dz \ \le \ \exp[-t\{1+\phi'(1)\}] \ .
\ee
 We conclude from (\ref{D1}), (\ref{D2}) that $\al=1$. 
\end{proof}
\begin{lem}
Let $w(x,t), \ x,t\ge 0$, be the solution to (\ref{A1}), (\ref{B1}) with coefficients satisfying (\ref{D1}), (\ref{E1}). Assume the initial data $w(\cdot,0)$ has beta function $\beta(\cdot,0)$ satisfying 
(\ref{U1}) with $0<\beta_0<1$. Then there is a positive constant $C$ depending only on the initial data such that $\kappa(t)\ge C$ for all $t\ge  0$. 
\end{lem} 
\begin{proof} 
Setting $c(x,t)=-\pa w(x,t)/\pa x\ge 0, \ 0\le x \le 1$, and $X_t$ to be the random variable with probability density function  $c(x,t)/w(0,t), \ 0\le x \le 1$, we see from (\ref{F1}) that $\kappa(t)$ satisfies the inequality 
\be \label{E2}
\kappa(t) \ \ge \  \frac{1+\phi'(1)}{\psi(0)} \langle \  X_t \ \rangle \ ,
\ee
where $\langle\cdot\rangle$ denotes expectation value. We assume that  $\beta_0=\lim_{x\ra 1}\beta(x,0)<1$.  Since the function $x\ra F(x,t), \ 0\le x<1,$ is convex,  it follows from the inequality (57) of  \cite{c} that $\beta(x,t)\le \beta(F(x,t),0)$ for $0\le x<1$.   Hence Lemma 2.1 implies that there exists $T>0$ depending only on the initial data,  such that $\beta(x,t) \le (1+\beta_0)/2, \ 0\le x < 1, \ t\ge T$.  Now for a positive random variable $X$ which has beta function $\beta(\cdot)$ and satisfies $\|X\|_\infty<\infty$, one  finds after integration by parts,
\be \label{F2}
\langle \ X \ \rangle  \ =  \ \|X\|_\infty-\int_0^{\|X\|_\infty} \beta(z) \  dz \ .
\ee
Applying (\ref{F2}) to  the variable $X_t$ with $t\ge T$, and using the fact that $\|X_t\|_\infty=1$, we conclude that $\langle \ X_t \  \rangle \ \ge \ (1-\beta_0)/2 $ provided $t\ge T$.  The result follows by observing that $\kappa(t)$ is a continuous strictly positive function of $t$ for $t\ge 0$.
\end{proof}
To obtain an upper bound on $\kappa(\cdot)$ we first obtain an alternative formula to (\ref{F1}) for $\kappa(t)$.  Observing that the function $c(\cdot,t)$ of (\ref{N1}) satisfies $c(x,t)=-\pa w(x,t)/\pa x\ge 0$, we see that $c(x,t)$ satisfies the equation
\be \label{G2}
 \frac{\pa c(x,t)}{\pa t}+ \frac{\pa }{\pa x} \left\{\left[\phi(x)-\kappa(t)\psi(x)\right] c(x,t)\right\}=c(x,t) \ .
 \ee
Hence we obtain a formula for $\kappa(t)$ equivalent to (\ref{F1}),
\be \label{N2}
\kappa(t)= \frac{\int_0^1 [x+\phi(x)] c(x,t) \  dx}{ \int_0^1 \psi(x)c(x,t) \  dx}=\frac{\langle \ X_t+\phi(X_t) \ \rangle}{\langle \ \psi(X_t) \  \rangle} \  .
\ee
\begin{lem} Let $X$ be a positive random variable such that $\|X\|_\infty=1$,  and set $\kappa(X)=\langle \ X+\phi(X) \ \rangle/\langle \ \psi(X) \ \rangle$ where $\phi(\cdot), \ \psi(\cdot)$ satisfy (\ref{D1}), (\ref{E1}).  Then for any $\del, \ 0<\del<1$,  there are positive constants $C_1(\del), \ C_2(\del)$ with the property $\lim_{\del\ra 0}C_1(\del)=\infty$ and  $\lim_{\del\ra 1}C_2(\del)=0$,   such that
\begin{eqnarray} \label{O2}
1-\langle X  \rangle  \ &\le& \ \del \quad {\rm implies} \quad \kappa(X) \ge C_1(\del), 
\\
1-\langle  X  \rangle  \ &\ge& \ \del \quad {\rm implies} \quad \kappa(X) \le C_2(\del). \label{P2}
\end{eqnarray} 
\end{lem}
\begin{proof}   We see from (\ref{E1}) that for any $\eta>0$,
\be \label{Q2}
\langle \ \psi(X) \ \rangle \  \le \ \psi(0)P(X<\eta) +|\psi'(\eta)|\left[1-\langle  X \rangle\right] \ .
\ee
Combining (\ref{Q2}) with the inequality
\be \label{R2}
P(X<\eta) \  \le \ [1-\langle X\rangle]/(1-\eta) \ , \quad 0<\eta<1,
\ee
we conclude that there is a constant $C>0$ depending only on  $\psi(\cdot)$  such that
\be \label{RR2}
\kappa(X) \ \ge  \ C\langle X\rangle/[1-\langle X\rangle] \ .
\ee
This proves  (\ref{O2}).

To prove (\ref{P2}) observe that by Jensen's inequality,  $\langle \ \psi(X) \ \rangle \ \ge \ \psi(\langle X\rangle) \ \ge \ \psi(1-\del)>0$ and $\langle  \ X+\phi(X) \ \rangle \ \le  \ \langle  X\rangle+\phi(\langle X\rangle) \ \le1- \del+\sup_{0\le x\le 1-\del}\phi(x)$. Now (\ref{P2}) and  $\lim_{\del\ra 1}C_2(\del)=0$ follows from the continuity of $\phi(\cdot)$ and the fact that $\phi(0)=0, \ \psi(0)>0$.
\end{proof}
\begin{lem} Let $w(x,t), \ x,t\ge 0$, be the solution to (\ref{A1}), (\ref{B1}) with coefficients satisfying (\ref{D1}), (\ref{E1}), (\ref{O1}). Assume the initial data $w(\cdot,0)$ has beta function $\beta(\cdot,0)$ satisfying  (\ref{U1}) with $0<\beta_0<1$.  Then there is a positive constant $C$  depending only on the initial data such that $\kappa(t)\le C$ for all $t\ge  0$.
\end{lem} 
\begin{proof}
From (\ref{N1}) we see that $h(x,t)$ satisfies $w(x,t)=-\pa h(x,t)/\pa x\ge 0, \ \lim_{x\ra 1}h(x,t)=0$,  whence it follows that $h(x,t)$ is a solution to  the equation
 \be \label{H2}
\frac{\pa h(x,t)}{\pa t}+ \left[\phi(x)-\kappa(t)\psi(x)\right] \frac{\pa h(x,t)}{\pa x}= \int_x^1 [\phi'(z)-\kappa(t)\psi'(z)]w(z,t) dz + h(x,t) \ .
\ee
We conclude then from (\ref{A1}), (\ref{G2}), (\ref{H2}), that the function $\beta(x,t)$ of (\ref{M1})  is a solution to
\be  \label{J2}
\frac{\pa}{\pa t} \log\beta(x,t)+ \left[\phi(x)-\kappa(t)\psi(x)\right] \frac{\pa}{\pa x} \log\beta(x,t) \ = \ -g(x,t) \ ,
\ee
 where the function $g(x,t)$ is given by the formula
 \be \label{K2}
g(x,t) \ = \ \{\phi'(x)-\kappa(t)\psi'(x)\}-\frac{1}{h(x,t)}\int_x^1 [\phi'(z)-\kappa(t)\psi'(z)]w(z,t) dz \ .
\ee  
It follows from (\ref{D1}), (\ref{E1}) and the non-negativity of $\kappa(\cdot)$ that $g(\cdot,\cdot)$ is a non-negative function and $\lim_{x\ra 1} g(x,t)=0$. From (\ref{K2}) we also have that
\be \label{L2}
\frac{\pa g(x,t)}{\pa x} \ = \  \{\phi''(x)-\kappa(t)\psi''(x)\}-\frac{w(x,t)}{h(x,t)^2}\int_x^1 [\phi''(z)-\kappa(t)\psi''(z)]h(z,t) dz \ .
\ee
Assuming now that $-\phi''(\cdot), \ \psi''(\cdot)$ are decreasing, it follows from (\ref{L2}) that 
\be \label{M2}
\frac{\pa g(x,t)}{\pa x} \ \le \  \{\phi''(x)-\kappa(t)\psi''(x)\}\left[ 1-\frac{w(x,t)}{h(x,t)^2}\int_x^1 h(z,t) dz\right] \ .
\ee
Note that the expression in square brackets on the RHS of (\ref{M2}) is $1$ minus the beta function of the convolution of $h(\cdot,t)$ with the function $H:\R\ra\R$ defined by $H(z)=0, \ z>0;  \ \ H(z)=1, \ z\le 0$. We observed in \cite{c} that if $\beta(\cdot)$ is the beta function associated with a function $h(\cdot)$ by (24) of \cite{c}, then the condition  $\sup\beta(\cdot)\le 1$ is equivalent to the condition that  $h(\cdot)$ is log-concave. Since the function $H(\cdot)$ is  log-concave, the Pr\'{e}kopa-Leindler inequality \cite{v} implies that if $\sup\beta(\cdot,t)\le 1$  then the convolution $h(\cdot,t)*H$ is also log-concave.  It follows that if $\sup\beta(\cdot,t)\le 1$,  then the expression in the square brackets on the RHS of (\ref{M2}) is non-negative.  We can see this directly by writing $h(x,t)=\exp[-q(x,t)], \ 0\le x<1,$ where the function $x\ra q(x,t)$ is increasing and convex with $\lim_{x\ra 1} q(x,t)=\infty$.  Then
\begin{multline} \label{CA2}
\frac{w(x,t)}{h(x,t)^2}\int_x^1 h(z,t) dz \ = \  \exp[q(x,t)]\frac{\pa q(x,t)}{\pa x} \int_x^1 \exp[-q(z,t)] \ dz
\\ \le  \exp[q(x,t)]\int_x^1\frac{\pa q(z,t)}{\pa z}  \exp[-q(z,t)] \ dz \ = \ 1.
\end{multline}
We conclude from (\ref{M2}), (\ref{CA2}) that if  $\sup\beta(\cdot,t)\le 1$ then $g(x,t)$ is a decreasing function of $x$ with $\lim_{x\ra 1} g(x,t)=0$. 

From Lemma 2.1 we see that there is a $T_0\ge 0$ such that $\sup\beta(\cdot,t)\le 1$ for $t\ge T_0$ and $\inf\beta(\cdot,T_0)=\beta_0>0$. Next let $\del_0>0$ have the property that the constant $C_1(\del)$ in Lemma 2.3 satisfies $C_1(\del_0)>\kappa_0=\phi'(1)/\psi'(1)$. Suppose now that 
\be \label{S2}
\int_0^1\beta(x,t) \ dx \  \le \ \del_0
\ee
for $t$ in the interval $T_1\le t\le T_2$, where $T_1\ge T_0$ and there is equality in (\ref{S2}) when $t=T_1$. We show that in this case there is a $\del_1>0$ such that 
\be \label{T2}
\int_0^1\beta(x,t) \ dx \  \ge \ \del_1, \quad T_1\le t\le T_2 \ .
\ee
The result follows from (\ref{T2}) and Lemma 2.3.

To prove (\ref{T2}) we use the fact that for $t\ge T_1$ one has
\be \label{U2}
\beta(x,t) \ = \ \exp\left[ - \int_{T_1}^t g(x(s),s) \ ds\right] \  \beta(x(T_1),T_1)\ ,
\ee
where $x(s), \ s\le t$, is the solution of (\ref{A2}) with terminal condition $x(t)=x$. Observe next that since $\kappa(s)\ge C_1(\del_0)>\kappa_0$ for $T_1\le s\le T_2$,  one has
\be \label{V2}
\kappa(s)\psi(z)-\phi(z) \  \ge \ [\kappa(s)-\kappa_0]|\psi'(1)|(1-z)>0, \quad T_1\le s  \le T_2, \ 0<z<1.
\ee
We conclude that
\be \label{W2}
[1-x(s)] \ \le \  [1-x]\exp\left\{-\int_s^t[\kappa(s'))-\kappa_0]|\psi'(1)|ds'\right\} \quad T_1\le s\le t\le T_2 \ .
\ee 
Observe now that for  any $s, \  T_1\le s  \le T_2,$ the function $\phi'(z)-\kappa(s)\psi'(z)$ is a positive decreasing  function of $z, \ 0<z<1$ and the function $g(\cdot,s)$ of (\ref{K2}) satisfies the inequality
\be \label{X2}
0 \ \le \ g(z,s) \  \le \  \phi'(z)-\kappa(s)\psi'(z), \quad T_1\le s  \le T_2, \ 0<z<1.
\ee
It follows from (\ref{W2}), (\ref{X2}) that
\be \label{Y2}
0 \ \le \  \int_{T_1\vee (t-1)}^{t} g(x(s),s) \ ds \  \le  C_3(\del_0), \quad  T_1\le t\le T_2,
\ee
for a constant $C_3(\del_0)$ depending only on $\del_0$. From (\ref{L2}) and the fact that  $-\phi''(\cdot), \ \psi''(\cdot)$ are decreasing we see that for any $x_1>0$, 
\be \label{Z2}
0 \ \le g(z,s) \ \le [\kappa(s)\psi''(x_1)-\phi''(x_1)](1-z), \quad x_1\le z \le 1, \ s\ge T_0.
\ee
Hence if $T_1<t<T_2$ then we have the inequality
\begin{multline} \label{AA2}
 \int_{T_1}^{T_1\vee(t-1)} g(x(s),s) \ ds \ \le \\
  \int_{T_1}^{T_1\vee(t-1)} ds \  [\kappa(s)\psi''(x_1)-\phi''(x_1)](1-x_1)
 \exp\left\{-\int_{s}^{T_1\vee(t-1)} [\kappa(s')-\kappa_0]|\psi'(1)| \ ds' \right\} \ ,
\end{multline}
where $x(t-1)=x_1\ge C_4(\del_0)$ for a positive constant $C_4(\del_0)$ depending only on $\del_0$.  It follows from (\ref{Y2}), (\ref{AA2}) that there is a constant $C_5(\del_0)$ depending only on $\del_0$ such that
\be \label{AB2}
0 \ \le \  \int_{T_1}^{t} g(x(s),s) \ ds \  \le \ C_5(\del_0), \quad T_1\le t\le T_2 .
\ee
We conclude then from (\ref{U2}) that there is a constant $C_6(\del_0)$ depending only on $\del_0$ such that
\be \label{AC2}
\beta(x,t) \  \ge \ C_6(\del_0) \  \beta(x(T_1),T_1) , \quad T_1\le t\le T_2.
\ee
In view of the monotonicity of the function $g(\cdot,s)$ for $s\ge T_0$ we also have that
\be \label{AD2}
\beta(z,s) \ \ge (1-\gamma) \beta(x,s), \quad s\ge T_0, \ z\ge x,
\ee
for some constant $\gamma<1$. Since $x(T_1)\ge x$ in (\ref{AC2}) we conclude from (\ref{AD2}) that
(\ref{T2}) holds.
\end{proof} 
\begin{lem} Under the conditions of Lemma 2.4  the limit (\ref{J1}) holds.
\end{lem} 
\begin{proof} It follows from (\ref{F2}) that if $X$ is a positive random variable with $\|X\|_\infty=1$ and beta function $\beta(\cdot)$ satisfying $\|\beta(\cdot)\|_\infty<1$ then $\langle X\rangle \ \ge \  1-\|\beta(\cdot)\|_\infty$.
As in Lemma 2.2 there exists $T_0\ge 0$ such that $\sup\beta(\cdot,t)\le (1+\beta_0)/2, \ t\ge T_0$.  Hence for $t\ge T_0$ there is the inequality $1\le w(0,t)\le 2/(1-\beta_0)$. Next for $0<\eta<\min[\beta_0/2,(1-\beta_0)/2]$ let $\ve(\eta)$ be such that $|\beta(x,T_0)-\beta_0|<\eta$ provided $1-x<\ve(\eta)$.  Then from  Lemma 1 of \cite{c} we see that there are constants $C_1(\eta), \ C_2(\eta)$ depending only on $\eta$ and $w(\cdot,0)$ such that
\be \label{AE2}
C_1(\eta) [1-x]^{(\beta_0+\eta)/(1-\beta_0-\eta)}   \      \le \ w(x,T_0)/w(0,T_0) \ \le \ C_2(\eta) [1-x]^{(\beta_0-\eta)/(1-\beta_0+\eta)}
\ee
provided $1-x<\ve(\eta)$.
Assuming now wlog that $T_0=0$, we see from  Lemma 2.1 that there exists $T_\eta\ge  0$ such that $1-F(0,t)<\ve(\eta)$ provided $t\ge T_\eta$. We conclude then from (\ref{AE2}) and the bound on $w(0,t)$ when $t\ge T_0$ the inequalities
\begin{eqnarray} \label{AF2}
w(0,0) C_2(\eta) e^t  [1-F(0,t)]^{(\beta_0-\eta)/(1-\beta_0+\eta)} \ & \ge & \ 1,   \quad t\ge T_\eta, \\
\label{AG2}
w(0,0)C_1(\eta) e^t [1-F(0,t)]^{(\beta_0+\eta)/(1-\beta_0-\eta)}  \ & \le& \ 2/(1-\beta_0) \ , \quad t\ge T_\eta.
\end{eqnarray}
Observe next from (\ref{B2}) using the convexity of the function $F(\cdot,t)$, that
\be \label{AH2}
1-F(0,t) \ \le \ \exp\left[ -\phi'(1)t+\psi'(1)\int_0^t\kappa(s) ds\right] \ .
\ee
Now (\ref{AF2}) and (\ref{AH2}) imply that
\be \label{AI2}
\limsup_{T\ra \infty}\frac{1}{T}\int_0^T \kappa(s) ds \ \le \  [1/\beta_0-\phi'(1)-1]/|\psi'(1)| \ .
\ee

In order to prove a lower bound on the time average of $\kappa(\cdot)$ analogous to (\ref{AI2}), we observe as in (\ref{D2}) that the solution $x(s), \ s\le t$, of (\ref{A2}) with terminal condition $x(t)=0$ satisfies the inequality
\be \label{AJ2}
\int_{x(t-\tau)}^1 w(z,t-\tau)\ dz \ \le \ \exp[-\tau\{1+\phi'(1)\}] \ , \quad 0<\tau<t.
\ee
We can also see as in (\ref{AE2}) that
\be \label{AK2}
w(z,s)/w(0,s) \  \ge \ C(1-z)^{(1+\beta_0)/(1-\beta_0)}, \quad 0<z<1, \ s\ge T_0,
\ee 
where the constant $C$ depends only on $\beta_0$. It follows then from (\ref{AJ2}), (\ref{AK2}) that there are positive  constants $C,\gamma$ depending only on $\beta_0$ such that
\be \label{AL2}
1-x(t-\tau) \  \le Ce^{-\gamma \tau}, \quad   t>T_0,  \ \tau<t-T_0.
\ee
If we use now (\ref{B2}), (\ref{AG2}) and (\ref{AL2}) we conclude the lower bound
\be \label{AM2}
\liminf_{T\ra \infty}\frac{1}{T}\int_0^T \kappa(s) ds \ \ge \  [1/\beta_0-\phi'(1)-1]/|\psi'(1)| \ .
\ee
\end{proof}
\begin{proof}[Proof of Theorem 1.1] This follows from Lemma 2.2, 2.4, 2.5.
\end{proof}

\section{The Quadratic Model}
We have already observed that the solution $w(x,t)$ of (\ref{A1}) is given by $w(x,t)=e^tw_0(F(x,t))$ where $F(x,t)$ is defined by (\ref{A2}). It follows from (\ref{A1}) that $F(x,t)$ is the solution to the initial value problem
\begin{eqnarray} \label{A3}
\frac{\pa F(x,t)}{\pa t}+ \left[\phi(x)-\kappa(t)\psi(x)\right] \frac{\pa F(x,t)}{\pa x}&=&0, \quad 0\le x<1, \ t\ge 0,
\\
F(x,0) &=& x, \quad 0\le x<1. \nonumber
\end{eqnarray}
Now suppose $ \phi(\cdot), \ \psi(\cdot)$ are quadratic and satisfy (\ref{D1}), (\ref{E1}). Then  $ \phi(\cdot), \ \psi(\cdot)$ are given by the formulas,
\be \label{B3}
\phi(x)=\phi'(1)x(x-1), \quad \psi(x)=\psi'(1)(x-1)+\psi''(1)(x-1)^2/2,
\ee
whence $\phi(\cdot), \ \psi(\cdot)$ are determined by the three parameters $\phi'(1), \psi'(1), \ \psi''(1)$, which are subject to the constraints in  (\ref{D1}), (\ref{E1}). For $t\ge 0$ let $u(t)$ be the function
\be \label{C3}
u(t)= \exp\left[\int_0^t \{\phi'(1)-\psi'(1)\kappa(s)\} \  ds \right] \ .
\ee
Then it is easy to see that if the function $v(t)$ is the solution to the initial value problem
\be \label{D3}
\frac{dv(t)}{dt}=u(t), \ t\ge 0, \quad v(0)=0,
\ee
the solution to (\ref{A3}) is given by the formula
\be \label{E3}
1-F(x,t)=\frac{1-x}{u(t)+a(t)(1-x)} \ , \quad 0\le x<1, \ t\ge 0, 
\ee
where $a(\cdot)$ is given in terms of $u(\cdot), \ v(\cdot)$ by the formula
\be \label{F3}
a(\cdot)=\{\psi''(1)[u(\cdot)-1]+ |\phi'(1)|[\psi''(1)-2\psi'(1)]v(\cdot)\}/2|\psi'(1)| \ .
\ee
Using the identity
\be \label{FF3}
|\phi'(1)|v(t) \ = \ 1-u(t) +|\psi'(1)|\int_0^t \kappa(s) u(s) \ ds \ , 
\ee
we see that $u(t)-1+|\phi'(1)|v(t)\ge 0$ for all $t$ since the function $\kappa(\cdot)$ is non-negative. Hence the function $a(\cdot)$ in (\ref{F3}) is strictly positive for all $t\ge 0$. 
Define now a function $G(u,v)$ by
\be \label{G3}
G(u,v)=\int_0^1 w_0\left(1-\frac{1-x}{u+a(1-x)}\right) \ dx,
\ee
with $a$ given in terms of $u,v$ by (\ref{F3}). Since the conservation law (\ref{B1}) is equivalent to $e^tG(u(t),v(t))=1$, it follows from (\ref{D3}) that
\be \label{H3}
G(u,v)+ G_u(u,v)\frac{du}{dt}+G_v(u,v)u=0 \ .
\ee
Hence if $[u(t),v(t)]$ is the solution to  the two dimensional dynamical system (\ref{D3}), (\ref{H3}) with initial condition $u(0)=1,v(0)=0$, then $w(x,t)=e^tw_0(F(x,t))$ with $F(x,t)$ given by (\ref{E3}) is the  solution to (\ref{A1}), (\ref{B1}) with initial condition $w_0(\cdot)$.

Observe now that $w_0(z)\sim (1-z)^p$ as $z\ra 1$ where $p={\beta_0/(1-\beta_0)}$, and also from Theorem 1.1 we have $\lim_{t\ra\infty} u(t)=\infty$. Hence from (\ref{F3}), (\ref{G3}) we may conclude that at large time,
\be \label{I3}
G(u,v) \ \sim \  u^{-p} G_0(v/u),
\ee
where the function $G_0(\xi)$ is given  by the formula
\be \label{J3}
G_0(\xi) \ = \  \int_0^1 \left[ \frac{1-x}{ 1+(1-x)\{a_1+a_2\xi\} }\right]^p \ dx  \ ,
\ee
with
\be \label{JJ3}
a_1 \ = \ \psi''(1)/2|\psi'(1)|, \quad a_2 \ = \ |\phi'(1)|[\psi''(1)-2\psi'(1)]/2|\psi'(1)|  \ .
\ee
Note that $a_1$ is non-negative and $a_2$  strictly positive. 
If we replace the function $G(u,v)$ of (\ref{G3}) by the RHS of (\ref{I3}), then we easily see that in the variables $[u,\xi]$ the system (\ref{D3}), (\ref{H3}) reduces to
\begin{eqnarray} \label{K3}
\frac{d\xi(t)}{dt} \ &=& \  \frac{(p-\xi)G_0(\xi)}{pG_0(\xi)+\xi G'_0(\xi)} \ , \\
\frac{d}{dt} \log u(t) \ &=& \  \frac{G_0(\xi)+G_0'(\xi)}{pG_0(\xi)+\xi G_0'(\xi)} \ . \label{L3}
\end{eqnarray}
It is evident from (\ref{K3}) that $\xi=p$ is a globally asymptotically stable critical point for the equation provided we can establish a few properties of the function $G_0(\cdot)$.
\begin{lem} The function $G_0(\xi)$ is a positive monotonic decreasing function of $\xi$ for $\xi>0$, and satisfies  the differential inequality
\be \label{M3}
\xi G_0'(\xi)+(p+1)G_0(\xi) \ \ge \ \left(1+\{\psi''(1)+|\phi'(1)|[\psi''(1)-2\psi'(1)]\xi\}/2|\psi'(1)|\right)^{-p} \ .
\ee
Furthermore the function $G_0(\cdot)$ satisfies the inequality
\be \label{N3}
G_0(\xi) \ < \ \left(1+\{\psi''(1)+|\phi'(1)|[\psi''(1)-2\psi'(1)]\xi\}/2|\psi'(1)|\right)^{-p} \ , \quad \xi \ge 0.
\ee
\end{lem}
\begin{proof} Observe from (\ref{J3}) that
\be \label{O3}
G_0(\xi) \ = \ \int_0^1[(1-x) g_0(x,\xi)]^p \  dx, \quad \xi\ge 0,
\ee
where
\be \label{P3}
0 \  \le \ -\frac{\pa}{\pa\xi} g_0(x,\xi) \ \le \  \frac{(1-x)}{\xi} \ \frac{\pa}{\pa x} g_0(x,\xi) \ .
\ee
The inequality (\ref{M3}) follows from (\ref{P3}) if we integrate by parts in (\ref{O3}). To see that (\ref{N3}) holds we use the fact that $(1-x) \pa g_0(x,\xi)/\pa x \  \le \ g_0(x,\xi)$, whence $(1-x)g_0(x,\xi)$ is a decreasing function.  Hence $G_0(\xi)\le g_0(0,\xi)^p$, which is (\ref{N3}).
\end{proof}
\begin{proposition} Assume that the functions $\phi(\cdot), \ \psi(\cdot)$ are quadratic and satisfy (\ref{D1}), (\ref{E1}). Assume further that the beta function $\beta(x,0), \ 0\le x\le 1$, for the initial data is H\"{o}lder continuous at  $x=1$ and $\beta(1,0)=\beta_0$ with $0<\beta_0<1$.  Then setting $\kappa=[1/\beta_0-\phi'(1)-1]/|\psi'(1)|$, there are positive constants $C,\gamma$ such that for $t\ge 0$
\be \label{Q3}
|\kappa(t ) -   \kappa|\le Ce^{-\gamma t}, \quad  \|\beta(\cdot,t)-\beta_\kappa(\cdot)\|_\infty\le Ce^{-\gamma t}, 
\ee
where $\beta_\kappa(\cdot)$ is the beta function of the time independent solution $w_\kappa(\cdot)$ of (\ref{G1}).
\end{proposition}
\begin{proof}   We write the function $G(u,v)$ of (\ref{G3})  as $G(u,v)=u^{-p}G_0(\xi,\eta)$, where $p=\beta_0/(1-\beta_0), \ \xi=v/u, \ \eta=1/u$. Thus $G_0(\xi,\eta)$ is given by the formula
\be \label{R3}
G_0(\xi,\eta) \ = \  \int_0^1 \eta^{-p}w_0\left(1- \frac{\eta z}{ 1+z\{a_1(1-\eta)+a_2\xi\} }\right) \ dz .
\ee
 With the extra dependence of $G_0(\cdot,\cdot)$, the system (\ref{K3}), (\ref{L3}) needs  to be modified to 
\begin{eqnarray} \label{S3}
\frac{d\xi(t)}{dt} \ &=& \  \frac{(p-\xi)G_0(\xi,\eta)+\eta \pa G_0(\xi,\eta)/\pa \eta}{pG_0(\xi,\eta)+\xi \pa G_0(\xi,\eta)/\pa \xi+\eta \pa G_0(\xi,\eta)/\pa \eta} \ , \\
\frac{d}{dt} \log u(t) \ &=& \  \frac{G_0(\xi,\eta)+ \pa G_0(\xi,\eta)/\pa \xi}{pG_0(\xi,\eta)+\xi \pa G_0(\xi,\eta)/\pa \xi+\eta \pa G_0(\xi,\eta)/\pa \eta} \ . \label{T3}
\end{eqnarray} 
Observe that the denominator on the RHS of (\ref{S3}), (\ref{T3}) is the same as $-u^{p+1}G_u(u,v)$ and hence by (\ref{G3}) is strictly positive.
From the proof of Lemma 2.5 it follows that for any $\delta>0$ there exists $T_\del>0$ such that
\be \label{U3}
u(t)\ge C_\del\exp[(1/\beta_0-1-\del)t], \quad t\ge T_\del \ ,
\ee
for a constant $C_\del$ depending on $\del$ and the initial data.  Choosing $\del<1/\beta_0-1$ in (\ref{U3}) we see that
the system (\ref{S3}), (\ref{T3}) converges  to the simpler system (\ref{K3}), (\ref{L3}) as $t\ra\infty$.

We first show that $\sup_{t\ge 0}\xi(t)<\infty$.  In the case $\beta(\cdot,0)\equiv\beta_0$ this follows from the inequality  $ |\pa G_0(\xi,\eta)/\pa \eta| \le pa_1G_0(\xi,\eta), \ \xi\ge 0, 0\le\eta\le 1$ and (\ref{U3}). More generally let $g_1(z), \  g_2(z,\xi,\eta)$ be defined by
\begin{eqnarray} \label{V3}
g_1(z) \ &=& \ \int_0^z [1-\beta(1-z',0)] \ dz',  \quad 0\le z <1,\\
g_2(z,\xi,\eta) \  &=& \  \frac{ z}{ 1+z\{a_1(1-\eta)+a_2\xi\} }, \quad 0\le z<1. \nonumber \ 
\end{eqnarray}
Then  $ \pa G_0(\xi,\eta)/\pa \eta$ is given by the formula
\be \label{W3}
\frac{\pa }{\pa \eta} G_0(\xi,\eta) \ = \  \int_0^1  g_3(z,\xi,\eta) \ \eta^{-p}w_0\big(1- \eta g_2(z,\xi,\eta)\big) \ dz ,
\ee
where 
\begin{multline} \label{X3}
g_3(z,\xi,\eta) \ = \ \left[   \frac{a_1\beta\big(1- \eta g_2(z,\xi,\eta), 0\big)\eta g_2(z,\xi,\eta)^2}{g_1\big(\eta g_2(z,\xi,\eta)\big)}  \right] + \\
\left[   \frac{\beta\big(1- \eta g_2(z,\xi,\eta), 0\big) g_2(z,\xi,\eta)}{g_1\big(\eta g_2(z,\xi,\eta)\big)}  -\frac{p}{\eta}\right]  \ .
\end{multline}
Observe that there exists $\eta_0>0$ such that the first term on the RHS of (\ref{X3}) and $\eta$ times the second term are bounded by a constant  for all $(z,\xi,\eta)$ with $0\le z\le1,\xi\ge 0, 0\le\eta\le \eta_0$. 
We conclude that $ |\eta \pa G_0(\xi,\eta)/\pa \eta| \le CG_0(\xi,\eta), \ \xi\ge 0, 0\le\eta\le \eta_0$,  for some constant $C$. Hence (\ref{S3}) implies that $\sup_{t\ge 0}\xi(t)<\infty$. 

Next we obtain bounds on the denominator of the RHS of (\ref{S3}), (\ref{T3}). The denominator is  $-u^{p+1}G_u(u,v)$, which is given in terms of the $(\xi,\eta)$ variables by
\be \label{Y3}
-u^{p+1}G_u(u,v) \ = \ \int_0^1  g_4(z,\xi,\eta) \ \eta^{-p}w_0\big(1- \eta g_2(z,\xi,\eta)\big) \ dz ,
\ee
where
\be \label{Z3}
 g_4(z,\xi,\eta)  \ = \   \frac{(1+a_1 z)\beta\big(1- \eta g_2(z,\xi,\eta), 0\big)\eta z^{-1}g_2(z,\xi,\eta)^2}{g_1\big(\eta g_2(z,\xi,\eta)\big)}   \ .
\ee
It is evident from (\ref{Z3}) that there exists $\eta_0>0$ such that for any $\xi_0\ge 0$, there are positive constants $C_1,C_2$ with the property
\be \label{AA3}
C_1 \ \le \ g_4(z,\xi,\eta)  \  \le C_2, \quad 0\le \eta\le \eta_0, \ 0\le \xi\le \xi_0, \ 0\le z\le 1.
\ee
It follows from (\ref{AA3}) that
\be \label{AB3}
C_1G_0(\xi,\eta) \ \le \  -u^{p+1}G_u(u,v)  \  \le C_2G_0(\xi,\eta), \quad 0\le \eta\le \eta_0, \ 0\le \xi\le \xi_0.
\ee

To see that $[\xi(t),\kappa(t)]$ converges exponentially fast to $[p,\kappa]$, we need to use the H\"{o}lder continuity of $\beta(x,0)$ at  $x=1$.  Observe that the H\"{o}lder continuity  implies that $\eta$ times the second term of (\ref{X3}) is bounded by $\eta^{\alpha}$ for some $\alpha>0$ when $\eta<<1$. The exponential convergence of $\xi(t)$ to $p$ as $t\ra\infty$ follows now from (\ref{S3}) and (\ref{AB3}). To see exponential convergence of $\kappa(t)$ we use the fact that  $ |\pa G_0(\xi,\eta)/\pa \xi| \le CG_0(\xi,\eta), \ \xi\ge 0, 0\le\eta\le 1$, for some constant $C$. The convergence follows then from the fact that 
$\lim_{\eta\ra 0} G_0(\xi,\eta)=G_0(\xi), \ \xi\ge 0$, Lemma 3.1,  the exponential convergence of $\xi(t)$ and (\ref{T3}).

To see that $\beta(\cdot,t)$  converges as $t\ra\infty$ first note that the invariant solution $w_\kappa(\cdot)$ of (\ref{G1}) with $\kappa>\phi'(1)/\psi'(1)$ is given by the formula
\be \label{WW3}
w_\kappa(x) \ = \  C\left[\frac{1-x}{1+(1-x)\{a_1+pa_2\}}\right]^p \ ,
\ee
for some positive constant $C$. It follows that $w(x,t)=w_\kappa(x)g(x,t)$ where  $g(x,t)$ is a positive function defined by
\be \label{AC3}
\frac{\pa}{\pa x} \log g(x,t) \ = \  \frac{p}{[1+z\{a_1+pa_2\}]z}- \frac{g_4(z,\xi(t),\eta(t))}{(1+a_1z)z}  \ , \quad z=1-x.
\ee
The H\"{o}lder continuity of $\beta(x,0)$ at  $x=1$ and (\ref{AC3}) implies that
\be \label{AD3}
|(1-x)\frac{\pa}{\pa x} \log g(x,t)| \ \le \ Ce^{-\gamma t}, \quad 0\le x\le 1, \ t\ge 0,
\ee
for some positive constants $C,\gamma$.  The exponential convergence of $\beta(\cdot, t)$ follows from (\ref{AD3}). To see this we note that
\begin{multline} \label{AE3}
|h(x,t)-h_\kappa(x) g(x,t)| \  \le \ \int_x^1 h_\kappa(x')|\pa g(x',t)/\pa x'|  \ dx' \\
\le \  Ce^{-\gamma t}\int_x^1 w_\kappa(x')g(x',t) \ dx' \ = \  Ce^{-\gamma t} h(x,t),
\end{multline}
where $h_\kappa(\cdot)$ is the $h$ function associated with $w_\kappa(\cdot)$. Similarly we have that
\be \label{AF3}
|c(x,t)-c_\kappa(x) g(x,t)| \  \le \  p^{-1}Ce^{-\gamma t} c_\kappa(x) g(x,t) \ ,
\ee
where $c_\kappa(\cdot)$ is the $c$ function associated with $w_\kappa(\cdot)$. 
\end{proof}
\begin{proof}[Proof of Theorem 1.2-subcritical case] The fact that $\lim_{t\ra\infty}\xi(t)=p$ follows from the argument of Proposition 3.1 on observing that continuity of $\beta(x,0)$ at $x=1$ implies \\
 $\lim_{\eta\ra 0} \eta \sup_{0\le \xi\le \xi_0}[|\pa G_0(\xi,\eta)/\pa \eta|/G_0(\xi,\eta)]=0$ for any $\xi_0\ge 0$. Now $\lim_{t\ra\infty}\kappa(t)=\kappa$ follows from  $\lim_{t\ra\infty}\xi(t)=p,\ \lim_{t\ra\infty}\eta(t)=0$ and (\ref{T3}). The convergence of $\beta(\cdot,t)$ to $\beta_\kappa(\cdot)$ in the $L^\infty$ norm as $t\ra \infty$ follows just as in Proposition 3.1 by noting that continuity of $\beta(x,0)$ at $x=1$ implies the inequality (\ref{AD3}) holds with a constant $C(t)$ on the RHS which has the property $\lim_{t\ra\infty}C(t)=0$.
\end{proof}

\section{The Critical Case}
Here we begin the proof of  Theorem 1.3 using only the beta function methodology.  First we consider a necessary condition obtained by Niethammer and Velasquez \cite{nv2} on the initial data $w(x,0), \ 0\le x<1$, of (\ref{A1}), (\ref{B1}) for convergence in the critical case to the self-similar solution  at large time.  We show that this condition, which was proven in Theorem 3.1 of \cite{nv2}, is implied by the condition $\lim_{x\ra 1}\beta(x,0)=1$.  The condition for convergence of \cite{nv2} is given in terms of a new variable $y$ determined by the requirement that $w_{\kappa_0}(x)/w_{\kappa_0}(0)=e^{-y}, \ 0\le x<1$. Writing  $w(x,0)=\tilde{w}_0(y), \ 0\le x<1$,  the necessary condition for convergence is that
\be \label{A4}
\lim_{y\ra\infty} \frac{\tilde{w}_0(y+\la(y)z)}{\tilde{w}_0(y)}          \ = \ e^{-z}
\ee
locally uniformly in $z\ge 0$ for some positive function $\la(y), \ y\ge 0$.
\begin{proposition} Suppose the initial data $w(x,0), \ 0\le x<1$, of (\ref{A1}), (\ref{B1}) satisfies  $\lim_{x\ra 1}\beta(x,0)=1$. Then (\ref{A4}) holds for the function $\la(y) \ = 2g(x)/[\kappa_0\psi''(1)-\phi''(1)](1-x)^2$,  where
\be \label{B4}
g(x) \ = \ \int_x^1[1-\beta(x',0)] \  dx' \ , \quad 0\le x<1.
\ee
\end{proposition} 
\begin{proof} We first observe that
\be \label{C4}
\lim_{x\ra 1}\frac{w(x+zg(x),0)}{w(x,0)} \ = \ e^{-z}
\ee
locally uniformly in $z\ge 0$. To see this note that the logarithm of the fraction on the LHS of (\ref{C4}) is given by $zg(x)$ times
\be \label{D4}
\frac{d}{dx'} \log w(x',0) \ = \ -\frac{\beta(x',0)}{g(x')} \ ,
\ee
for some $x'$ satisfying $x<x'<x+zg(x)$, and that
\be \label{E4}
|g(x)-g(x')| \  \le \ z g(x)\sup_{x\le x''<1}|1-\beta(x'',0)| \ .
\ee

Next we show that (\ref{C4}) implies (\ref{A4}). To do this we note that the transformation $x\ra y $ is explicitly given by
\be \label{F4}
y \ = \ \int_0^x\frac{dx'}{\left[\kappa_0\psi(x')-\phi(x')\right]} \ = \ \frac{2[1+o(1-x)]}{[\kappa_0\psi''(1)-\phi''(1)](1-x)} \  , 
\ee
assuming the continuity of $\phi''(x),\psi''(x)$ at $x=1$.  Suppose now that $x_z\ra y+\la(y)z$. Since the function $\kappa_0\psi(\cdot)-\phi(\cdot)$ is positive decreasing we conclude from (\ref{F4}) that
\be \label{G4}
\la(y)z \  \ge \  \frac{2[1+o(1-x)](x_z-x)}{[\kappa_0\psi''(1)-\phi''(1)](1-x)^2} \ .
\ee
Now (\ref{E4}) and the fact that $\lim_{x\ra 1}\beta(x,0)=1$ implies that $x_z-x\le (1-x)o(1-x)$, whence we obtain from (\ref{F4}) the upper bound
\be \label{H4}
\la(y)z \  \le \  \frac{2[1+o(1-x)](x_z-x)}{[\kappa_0\psi''(1)-\phi''(1)](1-x)^2} \ .
\ee
The result follows from (\ref{C4}), (\ref{G4}), (\ref{H4}).
\end{proof}
Next we wish to obtain a uniform upper  bound on $\kappa(t), \ t\ge 0$, in the critical case.  In view of (\ref{F2}) and Lemma 2.3, this is a consequence of the following: 
\begin{lem} Let $w(x,t), \ x,t\ge 0$, be the solution to (\ref{A1}), (\ref{B1}) with coefficients satisfying (\ref{D1}), (\ref{E1}), (\ref{O1}). Assume the initial data $w(\cdot,0)$ has beta function $\beta(\cdot,0)$ satisfying 
(\ref{U1}) with $0<\beta_0\le 1$.  Then  there are constants $\beta_\infty>0$ and $T_0\ge 0$  depending only on  the initial data, such that $\inf\beta(\cdot,t)\ge \beta_\infty$ for all $t\ge T_0$.
\end{lem}
\begin{proof}   We follow the argument of Proposition 10 of \cite{c}. Thus for $N=0,1,2,...,$ define points $x_N(0)$ by 
\be \label{L4}
x_0(0)=0, \quad w(x_N(0),0)=w(x_{N-1}(0),0)/2 \ {\rm for} \ N\ge 1.
\ee
Let $x_N(s), \ s\ge 0,$ be the solution of the differential equation (\ref{A2}) with initial condition $x_N(0)$.  Then there is an increasing function $\mathcal{N}:(0,\infty)\ra\Z^{+}$ such that $x_N(t)\ge 0$ for $N\ge\mathcal{N}(t),$ and $x_N(s)=0$ for some $s<t,$ if $N<\mathcal{N}(t)$. From Lemma 2.1 we see that $\lim_{t\ra\infty}\mathcal{N}(t)=\infty$.   For $t>0, \ N\ge\mathcal{N}(t),$ let  $I_N(t)$ be the interval $I_N(t)=\{x \  : \ x_N(t)\le x\le x_{N+1}(t)\}$ with length $|I_N(t)|$.  It follows from (\ref{A2}) that
\begin{multline} \label{M4}
|I_N(t)|\big/|I_N(0)| \ = \\
 \exp\left[ \int_0^t ds\int_0^1d\la \   \ \left\{ \phi'(\la x_N(s)+(1-\la)x_{N+1}(s))-\kappa(s) \psi'(\la x_N(s)+(1-\la)x_{N+1}(s))\right\}\right] \ .
\end{multline}
Hence from (\ref{D1}), (\ref{E1})  and (\ref{M4}) we conclude  that the ratio $|I_N(t)|/|I_{N+1}(t)|$ is an increasing function of $t$, and from \cite{c} that
\be \label{N4}
\lim_{N\ra\infty} |I_N(0)|\big/ |I_{N+1}(0)| \ = \ 2^{1/\beta_0-1} \ \ge \ 1 \ .
\ee

We define a function $\beta_N(t)$ for $t>0, \ N\ge\mathcal{N}(t)$ by
\be \label{O4}
\beta_N(t) \ = \ \exp\left[ \int_0^t ds \   \ |I_N(s)|\left\{ \phi''(x_{N+1}(s))-\kappa(s) \psi''(x_{N+1}(s))\right\}\right] \ ,
\ee
whence $\beta_N(t)$ is a positive decreasing function  of $t$ provided $\mathcal{N}(t)\le N$. From 
(\ref{D1}), (\ref{E1}), (\ref{O1}) it follows that there exists constants $C,\al$ satisfying $0<C,\al<1$ such that
\be \label{P4}
|I_N(t)| \big/|I_{N+1}(t)| \ \ge \ C\big/ \beta_N(t)^\al \quad {\rm for \ } t\ge 0. 
\ee
In view of (\ref{N4}) there exists $N_0\ge 0$ such that for $ N\ge \max\{N_0,\mathcal{N}(t)\},$
\be \label{Q4}
\beta_N(t) \ \le \ \exp\left[ \frac{1}{2}\int_0^t ds \   \ |I_{N+1}(s)|\left\{ \phi''(x_{N+2}(s))-\kappa(s) \psi''(x_{N+2}(s))\right\}\right] \ = \ \beta_{N+1}(t)^{1/2} \ .
\ee
We conclude from (\ref{P4}), (\ref{Q4}) that there exists $T_0\ge 0$ and a function $\mathcal{N}_1:[T_0,\infty)\ra\Z^+\cup\{\infty\}$ with the property   that  $\mathcal{N}_1(t)\ge\mathcal{N}(t)$ and $\beta_N(t)$ satisfies
\begin{multline} \label{R4}
\beta_N(t) \ \ge \ (C/2)^{3/\al} \ {\rm if \ } N\ge \mathcal{N}_1(t), \\
\beta_N(t) \quad {\rm is \ an \ increasing \  function \ of \ }N \quad  {\rm if \ } \mathcal{N}(t)\le N<\mathcal{N}_1(t) \ .
\end{multline}

As in \cite{c} we can compare the function $\beta(\cdot,t)$ to the functions $\beta_N(t), \ N\ge \mathcal{N}(t)$. For $0\le x<1$ let $I_x(t)=\{x': w(x,t)/2\le w(x',t)\le w(x,t)\}$, so that the left endpoint of the interval $I_x(t)$ is $x$ and $I_N(t)=I_{x_N(t)}(t)$. In view of (\ref{D1}), (\ref{E1}), (\ref{O1}) and the fact that $\beta_0\le 1$,  it follows that there are positive constants $\ga_1,\ga_2$ such that the function $g(\cdot,\cdot)$ defined by (\ref{K2}) satisfies the inequalities
\be \label{S4}
\ga_1|I_x(t)|[\kappa(t)\psi''(x+|I_x(t)|/2)-\phi''(x+|I_x(t)|/2)] \ \le \ g(x,t) \ \le \ 
\ga_2 |I_x(t)|[\kappa(t)\psi''(x)-\phi''(x)] \ .
\ee
It follows from (\ref{N4}), (\ref{O4}), (\ref{S4}) that there exists $\al, C>0$ and $T_0\ge 0,$ such that
\be \label{T4}
\beta(x,t) \ \ge \ C\beta_N(t)^\al \quad {\rm for \ }  x\in I_{N+1}(t), \ N\ge \mathcal{N}(t), \ t\ge T_0 \ .
\ee
We also conclude from (\ref{K2}), (\ref{S4}) that there exist positive constants $C,T_0,\al$ and 
\be \label{V4}
\beta(x',t) \ \ge \ C\beta(x,t)^{\al} \quad {\rm for \ } x'\in I_x(t), \  t\ge T_0. 
\ee
To see this we note that for $x\le x'\le x+|I_x(t)|/2$ the inequality (\ref{V4}) is a consequence of the fact that there exists a constant  $\ga>0$ such that
\be \label{U4}
\int_x^1 w(z,t) dz \ \le \  (1+\ga)\int_{x'}^1 w(z,t) dz \quad {\rm for} \ x\le x'\le x+|I_x(t)|/2 \ .
\ee
 For $x+|I_x(t)|/2\le x'\le x+|I_x(t)|$  the inequality follows from (\ref{S4}) since $|I_{x'}(t)|\le 2|I_x(t)|$ if $T_0$ is sufficiently large.
It follows  from (\ref{R4}), (\ref{T4}), (\ref{V4}) that there exist positive constants $\al, C,T_0$
such that
\be \label{W4}
\beta(x,t) \ \ge \ C\beta(0,t)^\al \quad {\rm for \ } 0\le x<1, \ t\ge T_0 \ .
\ee

We proceed now in a manner similar to that  followed in the proof of Lemma 2.4.
We choose $\del_0$ with $0<\del_0<1$ such that  $C[1-\del_0]/\del_0>\kappa_0=\phi'(1)/\psi'(1)$, where $C$ is the constant in (\ref{RR2}). We also choose $\del_1$ satisfying $\del_0<\del_1<1$ such that the constant $C_2(\del_1)$ of Lemma 2.3 satisfies the inequality $C_2(\del_1)<\kappa_0$.  Finally we choose $\beta_1$ with $0<\beta_1<1$ such that
\be \label{X4}
\beta_1\psi(0)\sup_{\del_0\le \del\le\del_1} C_2(\del)/[1-\del] \ \le \ 1/2 \ . 
\ee
With $T_0$ as in (\ref{W4}) and assuming $\beta_1>0$ sufficiently small,  we may suppose that $T_0\le T_1<T_2$  are such that $\beta(0,T_1)=\beta_1$ and $\beta(0,t)<\beta_1$ for $T_1<t<T_2$. Let $T_3$ satisfy $T_1\le T_3\le T_2$ and have the property that  $\langle X_{t}\rangle \ \ge \ 1-\del_0$ for $T_1<t\le T_3$ and either $T_3=T_2$ or  $\langle X_{T_3}\rangle=1-\del_0$. It follows from (\ref{RR2}),  (\ref{AB2}), and (\ref{W4})  that there is a constant $C_1$ such that
\be \label{Y4}
\beta(0,t) \ \ge \ C_1\beta_1 \ , \quad T_1\le t\le T_3 \ .
\ee
To obtain a lower bound for $\beta(0,t)$ in the region $T_3\le t\le T_2$ we use the equation
\be \label{Z4}
\frac{d}{dt}\log   \ \langle X_t\rangle \ = \  \frac{\beta(0,t)\psi(0)\kappa(t)}{\langle X_t\rangle}-1 \ .
\ee
Since $\langle X_{T_3}\rangle=1-\del_0$, it follows from (\ref{X4})  that
\be \label{AA4}
\frac{d}{dt}\log \langle X_t\rangle \ \le \  -\frac{1}{2} \ , \quad T_3\le t \le \  T_4,
\ee
where $T_4$ has the property that $\langle X_t\rangle \ \ge \  1-\del_1$ for $T_3\le t\le T_4$ and $\langle X_{T_4}\rangle=1-\del_1$ or $T_4=T_2$. Evidently (\ref{AA4}) implies that
\be \label{AB4}
T_4-T_3 \ \le \ 2\log[(1-\del_0)/(1-\del_1)]; \quad \langle X_t\rangle \ \le \ 1- \del_0 \ {\rm for \ } T_3\le t\le T_2 \ .
\ee
From Lemma 2.3 and (\ref{AA4}) we have that $C_2(\del_1)\le \kappa(s)\le C_1(\del_0)$ for $T_3\le s\le T_4$. Hence there is a constant $C_1(\del_0,\del_1)$ such that
\be \label{AC4}
0 \ \le \ \int_{T_3}^t g(x(s),s) \ ds  \ \le  C_1(\del_0,\del_1),
\ee
on any solution of (\ref{A2}) with $x(t)=0$, where $T_3\le t\le T_4$. We conclude from (\ref{Y4}), (\ref{AC4}) that $\beta(0,t)\ge  C_2(\del_0,\del_1)\beta_1$ for $T_3\le t\le T_4$. 

Finally we consider the interval $T_4\le t\le T_2$. From (\ref{Z4}) and the assumption $\beta(0,t)<\beta_1$ it follows that $\langle X_t\rangle \ \le \ 1-\del_1$ for $T_4\le t\le T_2$. Assuming that $\del_1>1/2$, we see from (\ref{F2}) and the fact that $\beta_0\le 1$ that there exists $x_1$ such that $0<x_1<3(1-\del_1)$ and  $\beta(x_1, t)\ge 1/2$. Let  $x_0>0$ be the unique maximum of the function $\phi(x)$ in the interval $0<x<1$.  In addition to choosing $\del_1>1/2$ such that $C_2(\del_1)<\kappa_0$, we choose it sufficiently close to $1$ so that  $3(1-\del_1)<x_0$. Observe now that since $\beta(x_1, t)\ge 1/2$ it follows that
\be \label{AD4}
0 \ \le \ \int_{T_4}^t g(x_1(s),s) \ ds  \ \le  \log 2 \ ,
\ee
on any solution of (\ref{A2}) with $x_1(t)=x_1$, where $T_4\le t\le T_2$. Letting $x_2(\cdot)$ be the solution of (\ref{A2}) with $x_2(t)=0$, it follows from the fact that $3(1-\del_1)<x_0$, that
\be \label{AE4}
0 \ \le \ \int_{T_4}^t [g(x_2(s),s)-g(x_1(s),s)] \ ds  \ \le   C(\del_1) \ , 
\ee
for a constant $C(\del_1)$ depending only on $\del_1$. We conclude from (\ref{AD4}), (\ref{AE4}) that
$\beta(0,t)\ge  C_2(\del_0,\del_1)\beta_1$ for $T_4\le t\le T_2$. 

We have therefore proven  that there is a constant $C$ such that  $\beta(0,t)\ge C\beta_1$ for $T_1\le t\le T_2$.  We conclude that $\inf_{t\ge T_0} \beta(0,t)>0$, whence the result follows from (\ref{W4}).
\end{proof}
\begin{corollary}
Suppose that the function $\phi(\cdot)$, in addition to satisfying the assumptions of Lemma 4.1, also satisfies the condition $\lim_{x\ra 0} \phi(x)/x \ = \ \infty$ . 
Then there is a positive constant $C$ depending only on the initial data $w(x,0), \ 0\le x< 1,$ for (\ref{A1}), (\ref{B1}) such that $\kappa(t)\ge C$ for all $t\ge 0$. 
\end{corollary}
\begin{proof}
From (\ref{N2}), (\ref{Z4}) we have that
\be \label{AF4}
\frac{d}{dt}\log   \ \langle X_t\rangle \ = \  \frac{\beta(0,t)\psi(0)}{\langle\psi(X_t)\rangle}\left[\frac{\langle\phi(X_t)\rangle}{\langle X_t\rangle}+1\right]-1 \ \ge \frac{ \beta_\infty \langle\phi(X_t)\rangle}{\langle X_t\rangle}-1 \ ,
\ee
for some $\beta_\infty>0$.  From Lemma 1 of \cite{c}  there exists a constant $\ga>0$  such that for $t\ge 0$ one has the inequality  $P(X_t>\ga\langle X_t\rangle)\ge 1/2$.  Hence we have that
\begin{multline} \label{AG4}
\langle\phi(X_t)\rangle-\phi'(1)\langle X_t\rangle \ =   \ \int_0^1 [\phi'(x)-\phi'(1)] P(X_t>x) \  dx \\
 \ge \ \frac{1}{2}\left[\phi(\ga \langle X_t\rangle)-\phi'(1)\ga\langle X_t\rangle\right] \ .
\end{multline}
It follows from (\ref{AF4}), (\ref{AG4}) that
\be \label{AH4}
\frac{d}{dt}\log   \ \langle X_t\rangle \ \ge \  \frac{ \beta_\infty \phi(\ga\langle X_t\rangle)}{2\langle X_t\rangle}+\phi'(1)[1-\ga/2]-1 \ ,
\ee
whence we conclude that there exists a positive constant $C$ such that $\langle X_t\rangle \ge C$ for $t\ge 0$. The result follows from (\ref{E2}). 
\end{proof}
The following proposition shows that if we assume  $\lim_{x\ra 0}\phi(x)/x<\infty$ then the lower bound of Corollary 4.1 may not hold for all initial data satisfying (\ref{U1}) with $\beta_0=1$. 
\begin{proposition}
Assume $\beta_0$ satisfies the inequality $0<\beta_0<1/[1+\phi'(0)]$, and  $w(x,t)$ is as in Lemma 4.1 with initial condition $w(x,0)=C(x_0-x)^{\beta_0/(1-\beta_0)}, \ 0\le x\le x_0; \quad w(x,0)=0$ for $x_0\le x\le 1$. There exists $\del(\beta_0)>0$ such that if $0<x_0<\del(\beta_0)$ then 
$\lim_{t\ra\infty} \kappa(t) \ = \ 0$ .
\end{proposition}
\begin{proof}
First observe that the linear approximation at $0$ to $\phi(x)-\kappa(t)\psi(x)$ is $\phi'(0)x-\kappa(t)\psi(0)$. The function $w(x,t)$ defined by
\begin{eqnarray} \label{AI4}
w(x,t)&=&Ce^{\la t}[x_0-xe^{\la t}]^{\beta_0/(1-\beta_0)} \quad {\rm for \ }  0\le x\le x_0 e^{-\la t} \ , \\ 
w(x,t) &=& 0 \qquad \qquad \qquad \qquad \qquad  \ \  {\rm for \ }  x_0 e^{-\la t}\le x\le 1  \ ,
\end{eqnarray}
is a solution to (\ref{A1}), (\ref{B1}) in this linear approximation provided $\la=1-\beta_0[1+\phi'(0)]>0$. In that case $\kappa(t)$ is given by the formula
\be \label{AJ4}
\kappa(t) \ = \ e^{-\la t}[1+\phi'(0)] (1-\beta_0)x_0/\psi(0) \ .
\ee

To prove that $\lim_{t\ra\infty} \kappa(t)=0$ more generally, one uses the equation (\ref{AF4}). From the argument  of Lemma 2.3 we see that
\be \label{AK4}
\frac{d}{dt}\log \langle X_t\rangle \ \le \  \frac{\beta(0,t)\psi(0)[\langle X_t\rangle+\phi(\langle X_t\rangle)]}{\langle X_t\rangle\psi(\langle X_t\rangle)}-1 \ .
\ee
The result follows from (\ref{AK4}) and Lemma 2.3 since $\beta(0,t)\le \beta_0$ for all $t\ge 0$. 
\end{proof} 
\begin{rem}
It is easy to construct initial data $w(x,0), \ 0\le x\le 1,$ for (\ref{A1}), (\ref{B1}) with support equal to the full interval $[0,1]$, the property
$\lim_{x\ra 1}\beta(x,0)=1$, and such that $w(\cdot,0)$ is arbitrarily close to the initial data of Proposition  4.2. In fact we can define $\beta(x,0)$ by
\be \label{AL4}
\beta(x,0)=\beta_0 \  \ {\rm for \ } 0\le x\le x_0, \quad \beta(x,0)=1-\ve(1-x) \  \ {\rm for \ } x_0<x\le 1, 
\ee
where $\ve<<1$. Note in this case the discontinuity in $\beta(x,0)$ at $x=x_0$.  In $\S6$ we are able to obtain a positive lower bound on $\inf\kappa(\cdot)$ for such initial data since $\beta(x,0)\le 1$ for $x$ close to $1$. We are not however able to obtain a lower bound if $\beta(x,0)$ oscillates above and below $1$ as $x\ra 1$. 
\end{rem}
\begin{lem}
Let $w(x,t), \ x,t\ge 0$, be the solution to (\ref{A1}), (\ref{B1})  with coefficients satisfying (\ref{D1}), (\ref{E1}), (\ref{O1}). Assume the initial data $w(\cdot,0)$ has beta function $\beta(\cdot,0)$ satisfying  (\ref{U1}) with $\beta_0=1$.  Then the limit (\ref{J1}) holds provided $\inf\kappa(\cdot)>0$ and
\be \label{AM4}
\inf_{t\ge 0} w(x,t)>0 \quad  {\rm  for \  all \ } x {\rm \ satisfying \ } 0\le x<1.
\ee
\end{lem} 
\begin{proof}  We define a function $z(t), \ t\ge 0$, by $e^t w_0(z(t))=1$. Since the conservation law (\ref{B1}) implies that $w(0,t)\ge 1$ we conclude that $z(t)\ge F(0,t), \ t\ge 0$. Observe also from (\ref{D4}) that $z(t)$ satisfies the differential equation
\be \label{AN4}
\frac{d z(t)}{dt} \ = \ \frac{g(z(t))}{\beta(z(t),0)}, \quad t\ge 0,
\ee
where $g(\cdot)$ is the function (\ref{B4}). Next we have from (\ref{AH2}) that
\be \label{AO4}
1-z(t) \ \le \ \exp\left[ -\phi'(1)t+\psi'(1)\int_0^t\kappa(s) ds\right] \ , \quad t\ge 0.
\ee
Since $\lim_{x\ra 1} \beta(x,0)=1$ it follows from (\ref{B4})  that
\be \label{AP4}
\lim_{t\ra\infty} \log[1-z(t)]/t \ = \ 0.
\ee
Hence we obtain the  upper bound (\ref{AI2}) in the case $\beta_0=1$.

To prove the lower bound (\ref{AM2}) for $\beta_0=1$ we first note that Lemma 2.3 implies that there is a positive constant $t_0$ depending only on the initial data such that $ \langle X_t\rangle \ \ge e^{-t_0}$ for all $t\ge 0$, whence $e^{t-t_0}w_0(F(0,t))\le 1$. We conclude that  $z(t-t_0)\le F(0,t)$ for all $t\ge 0$. The final fact we need in analogy to (\ref{AL2}) is that for any $\ve>0$ there exists $\del>0$ depending only on the initial data $w_0(\cdot)$ such that for any $t\ge 0$,
\be \label{AQ4}
\int_x^1 w(z,t) \ dz \ < \ \del \quad {\rm implies \ } 1-x<\ve \ .
\ee
It is easy now to conclude (\ref{AM2}) for $\beta_0=1$. Finally we note that (\ref{AM4}) implies (\ref{AQ4}).
\end{proof}

\vspace{.1in}
\section{The Quadratic Model-Critical Case}
We return to the quadratic model studied in $\S3$.
\begin{lem}
Assume the initial data $w_0(\cdot)$ for (\ref{A1}), (\ref{B1}) satisfies $\lim_{x\ra 1}\beta(x,0)=1$ and $w(x,t)=e^tw_0(F(x,t))$, where $F(x,t)$ is given by the formula (\ref{E3}). Then $\lim_{t\ra\infty} u(t)/v(t)=0$ if and only if there are constants $C_1,C_2>0$ such that $C_1\le\kappa(t)\le C_2$ for all $t\ge 0$.  
\end{lem}
\begin{proof}
We first assume $C_1\le\kappa(\cdot)\le C_2$, whence Lemma 2.3 implies that there exists  $C_3>0$ such that $\langle X_t\rangle\ge C_3$ for all $t\ge 0$.  We conclude then from Lemma 1 of \cite{c} that there exists $\gamma>0$ such that
\be \label{A5}
w(\gamma,t)/w(0,t) \ \ge \ 1/e,  \quad 0\le t<\infty .
\ee
Next we write
\be \label{B5}
w(\gamma,t)/w(0,t) \ = \ w_0(F(0,t)+[F(\gamma,t)-F(0,t)])/ w_0(F(0,t)) \ .
\ee
Since $\lim_{t\ra\infty} F(0,t)=1$, it follows from (\ref{C4}), (\ref{A5}) that there exists $T_0\ge 0$ such that
\be \label{C5}
F(\gamma,t)-F(0,t) \ \le \ 2g(F(0,t)) \ , \quad t\ge T_0 \ .
\ee
Using the fact that $\lim_{x\ra 1}\beta(x,0)=1$, we conclude that
\be \label{D5}
\lim_{t\ra\infty} \frac{F(\gamma,t)-F(0,t)}{1-F(0,t)} \ = \ 0 \ .
\ee
We see from the identity
\be \label{E5}
\frac{F(\gamma,t)-F(0,t)}{1-F(0,t)} \ = \  \frac{\gamma u(t)}{u(t)+(1-\gamma)a(t)} \ ,
\ee
and (\ref{D5}) that $\lim_{t\ra\infty}u(t)/a(t)=0$. Since (\ref{FF3}) implies that
\be \label{F5}
a(t) \  \le \{\psi''(1)\sup\kappa(\cdot)+2|\phi'(1)|\}v(t)/2 \ ,
\ee
we conclude that $\lim_{t\ra\infty} u(t)/v(t)=0$. 

Conversely let us assume that $\lim_{t\ra\infty} u(t)/v(t)=0$. Since $\lim_{t\ra\infty} F(0,t)=1$ we also have that $\lim_{t\ra\infty}[u(t)+a(t)]=\infty$, and hence we conclude that $\lim_{t\ra\infty} v(t)=\infty$. We define now $y(t)$ by
\be \label{G5}
y(t)= 1-1/a_1(t)=1- 2|\psi'(1)|/\{ \ |\phi'(1)| [\psi''(1)-2\psi'(1)]v(t)-\psi''(1) \ \} \ ,
\ee
and observe that $y(t)$ is an increasing function of $t$ which satisfies
\be \label{H5}
 \lim_{t\ra\infty} y(t)=1, \quad \frac{dy(t)}{dt} \ = \ \frac{\{ \ |\phi'(1)| \psi''(1)/2|\psi'(1)|+1\}u(t)}{a_1(t)^2} \ .
\ee
One can further  see that
\be \label{I5}
F(x,t)-y(t) \ = \ \frac{\{1+(1-x)\psi''(1)/2|\psi'(1)|\}u(t)}{a_1(t)[a_1(t)(1-x)+\{1+(1-x)\psi''(1)/2|\psi'(1)|\}u(t)]} \ ,
\ee
and hence we conclude that there are positive constants $C,T_0$ such that
\be \label{J5}
 F(x,t)-y(t) \ \le \ C\frac{dy(t)}{dt} \ , \quad {\rm for \ } t\ge T_0, \ 0\le x\le 1/2.
\ee
Let $z(t), \ t\ge0,$ be as in Lemma 4.2, whence $F(0,t)\le z(t), \ t\ge 0$.  Suppose now that at some $t\ge T_0$ one has $y(t)=z(t-\tau_0)$ where $\tau_0>0$. Then for $0\le x\le 1/2$ we have that
\be \label{K5}
e^{t}w_0(F(x,t)) \ \ge \ \e^{t} w_0\left(y(t)+C\frac{dy(t)}{dt}\right) \ 
=  \ e^{\tau_0} \frac{w_0(y(t)+Cdy(t)/dt)}{w_0(y(t))} \ .
\ee
Since (\ref{B1}) implies that the LHS of (\ref{K5}) is bounded above by $2$ when $x=1/2$, we conclude from (\ref{C4}) that if $\tau_0\ge 1+2C+\log 2$ and $T_0$ is sufficiently large then $dy(t)/dt \ge 2g(y(t))$. Hence if $y(T_0)\ge z(T_0-\tau_0)$ then $y(t)\ge z(t-\tau_0)$ for all $t\ge T_0$. Since (\ref{I5}) implies that $y(t)<F(0,t)$ we further have that $z(t-\tau_0)\le y(t)\le z(t)$ for $t\ge T_0$. 
We conclude therefore that
\be \label{L5}
\langle X_t\rangle \ = \ \frac{1}{e^tw_0(F(0,t))} \ \ge \ \frac{1}{e^t w_0(y(t))} \ \ge \ e^{-\tau_0} \ ,  \quad t\ge T_0.
\ee
Now Lemma 2.3 and (\ref{L5}) imply that $\inf \kappa(\cdot)>0$. 

To see that $\sup \kappa(\cdot)<\infty$,  we observe from (\ref{H5}), (\ref{I5}) that there are positive constants $\al,\beta$ with the property
\be \label{M5}
F(x,t) \ = \ y(t)+  \frac{\al+\beta(1-x)}{(1-x)+o(t)}\frac{dy(t)}{dt}, \quad 0\le x<1,
\ee
where $\lim_{t\ra\infty} o(t)=0$. It is easy to see that there exists $T_0>0$ such that $y(t)<F(0,t)<z(t)$  for $t\ge T_0$, whence  $y(t)=z(t-\tau(t))$ for some unique $\tau(t)> 0$. We show there are constants $\tau_1,\tau_2>0$ such that
\be \label{N5}
 \tau_1\le\tau(t) \ \le  \tau_2, \qquad t\ge T_0  \ .
\ee
To obtain the upper bound in (\ref{N5}) note that from (\ref{B1}),  (\ref{C4}), (\ref{M5}) there exists $\tau_2>0$ and  $T_1\ge T_0$ with the property that
\be \label{O5}
\tau(t)\ge \tau_2 \quad {\rm implies \ } \frac{dy(t)}{dt}\ge 2g(y(t)) \quad {\rm for \ } t\ge T_1.
\ee
Hence if $t_2\ge T_1$ and $\tau(t_2)>\tau_2$ then from (\ref{AN4}) and (\ref{O5}) we see that for sufficiently large $T_1$ and  $t\ge t_2$ satisfying  $\inf_{t_2\le s\le t} \tau(s)\ge \tau_2$, then
\be \label{P5}
y(t) \ \ge \ z( 3(t-t_2)/2 +t_2-\tau(t_2)) \quad {\rm which \  implies \ } \tau(t)\le \tau(t_2)-(t-t_2)/2 \ .
\ee
The upper bound in (\ref{N5}) follows. To obtain the lower bound observe again from (\ref{B1}), (\ref{C4}), (\ref{M5}) that there exists $\tau_1>0$ and  $T_1\ge T_0$ with the property that
\be \label{Q5}
\tau(t)\le \tau_1 \quad {\rm implies \ } \frac{dy(t)}{dt}\le g(y(t))/2 \quad {\rm for \ } t\ge T_1.
\ee
The lower bound in (\ref{N5}) follows from (\ref{Q5})  by analogous argument  for the upper bound.  

Assuming (\ref{N5}) holds, we show there exists $T_2\ge T_0$ and $\del>0$  such that
\be \label{R5}
e^tw_0(F(0,t)) \ \ge \ 1+\del, \quad t\ge T_2.
\ee
Thus from (\ref{C4}), (\ref{M5}) we see that for any $\eta$ with $0<\eta<1$ there exists $T_\eta\ge T_0$ such  that
\be \label{S5}
\int_0^{1-\eta} \exp\left[ -\frac{\al+\beta(1-x)}{(1-x)g(y(t))}\frac{dy(t)}{dt} \right] \ dx \ \le  \ e^{-\tau_1/2} \quad {\rm for \ } t\ge T_\eta \ .
\ee
Choosing $\eta<[1-e^{-\tau_1/2}]/2$ in (\ref{S5}) and putting $T_2=T_\eta$, we see that there is a constant $C(\tau_1)>0$ depending only on $\tau_1$ such that
\be \label{T5}
\frac{dy(t)}{dt} \ \ge \ C(\tau_1)g(y(t)) \quad {\rm for \ } t\ge T_2 \ .
\ee
Now (\ref{C4}), (\ref{M5}) and (\ref{T5}) imply that there exists $\del>0$ such that
\be \label{U5}
w_0(F(1/2,t)) \ \le \ \frac{1-\del}{1+\del} \ w_0(F(0,t)) \quad {\rm for \ } t\ge T_2 \ .
\ee
The inequality (\ref{U5}) and (\ref{B1}) imply (\ref{R5}). Since (\ref{R5}) implies that $\langle X_t\rangle \le  1/(1+\del)<1$ for $ t\ge T_2$, we see from Lemma 2.3 that $\sup\kappa(\cdot)<\infty$. 
\end{proof}
\begin{lem}
Assume the initial data $w_0(\cdot)$ for (\ref{A1}), (\ref{B1}) satisfies $\lim_{x\ra 1}\beta(x,0)=1$ and $w(x,t)=e^tw_0(F(x,t))$, where $F(x,t)$ is given by the formula (\ref{E3}). Then $\lim_{t\ra\infty} u(t)/v(t)=0$.
\end{lem}  
\begin{proof}
Observe that since $v(t)$ is an increasing function one has $\lim_{t\ra\infty} v(t)=v_\infty$ where $0<v_\infty\le \infty$.  If $v_\infty<\infty$ then it follows from (\ref{D3})  that there is an increasing sequence $t_m$ with $\lim_{m\ra\infty} t_m=\infty$ and  $u(t_m)\le 1$. In that case (\ref{E3}) implies that $\liminf_{t\ra\infty}  F(0,t)<1$, which is a contradiction to Lemma 2.1. We conclude that $\lim_{t\ra\infty} v(t)=\infty$. 

Next we show that there exist constants  $C_0,  \ T_0>0$ such that $u(t)\le C_0v(t)$ for all $t\ge T_0$. To see this we set $\xi(t)=v(t)/u(t)$ and note from (\ref{D3}), (\ref{H3}) that
\be \label{V5}
\frac{d\xi(t)}{dt} \ = \ \frac{u(t)G_u(u(t),v(t))+v(t)G_v(u(t),v(t))+\xi(t)G(u(t),v(t))}{u(t)G_u(u(t),v(t))} \ .
\ee 
Arguing as in Proposition 3.1, we see that there exists $v_0>0$ such that
\be \label{W5}
uG_u(u,v)+vG_v(u,v)+G(u,v)<0 \quad {\rm for \ } u>0,  \ v\ge v_0.
\ee
Hence there exists $T_0>0$ such that for any $t\ge T_0$ the function $v(t)/u(t)$ is increasing if $u(t)>v(t)$, whence there is a constant $C_0>1$ such that $u(t)\le C_0v(t)$ for $t\ge T_0$. 

It follows now from (\ref{I5}), (\ref{M5}) that there exists $T_1>0$ and a constant $C_1>0$ such that $o(t)$ in (\ref{M5}) satisfies the inequality $0\le o(t)\le C_1$ for $t\ge T_1$. 
Using the fact that $o(t)\ge 0$ we see from the argument to prove (\ref{N5}) that we can choose $T_1\ge T_0$ such that  $\tau(t)\le \tau_2$ for $t\ge T_1$. From (\ref{B1}) and the inequality $o(t)\le C_1$   we can further choose $T_2\ge T_1$ and $C_2>0$ such that for any $t\ge T_2$,  
\be \label{X5}
\frac{dy(t)}{dt} \ \le \ C_2g(y(t)) \ . 
\ee
The result follows from (\ref{B4}) and (\ref{X5}) since $\lim_{x\ra 1}\beta(x,0)=1$. 
\end{proof}
\begin{proof}[Proof of Theorem 1.2-critical case]
Using the notation of Lemma 5.1,  we shall show that there exists $\tau_0>0$ such that $\lim_{t\ra\infty} \tau(t)=\tau_0$. To obtain a formula for $\tau_0$ we assume $y(t)\sim z(t-\tau_0)$ and conclude from (\ref{AN4}) and (\ref{M5})  that for large $t$  
\be \label{Y5}
F(x,t) \ \sim \ z(t-\tau_0)+ \left[\frac{\al}{1-x}+\beta\right] g(z(t-\tau_0))\  , \quad 0\le x<1.
\ee
Now (\ref{B1}) and (\ref{C4}) imply that
\be \label{Z5}
e^{\tau_0-\beta}\int_0^1\exp\left[-\frac{\al}{1-x}\right]  \ dx \ = \ 1,
\ee
which uniquely determines $\tau_0>0$. 

We first prove that  $\liminf_{t\ra\infty} \tau(t)\le \tau_0$. To see this observe from (\ref{B1}), (\ref{C4}) and  (\ref{M5})  that if 
$\liminf_{t\ra\infty} \tau(t)\ge \tau_0+\ve$ for some $\ve>0$, then there exists $T_\ve$ sufficiently large and $\del(\ve)>0$ depending on $\ve$
with the property
\be \label{AA5}
\frac{d y(t)}{dt} \ \ge \ [1+\del(\ve)]g(y(t)) \ , \quad t\ge T_\ve .
\ee
Since $\lim_{x\ra 1}\beta(x,0)=1$ it follows from (\ref{AN4}) and (\ref{AA5})  that if $T_\ve$ is sufficiently large   depending only on $\ve$, then $y(t)\ge z([1+\del(\ve)/2](t-T_\ve)+T_\ve-\tau(T_\ve))$ for $t\ge T_\ve$.  Evidently this inequality implies that $\tau(t)\le 0$ for large $t$, which is a contradiction, whence  $\liminf_{t\ra\infty} \tau(t)\le \tau_0$.
We can further see that $\limsup_{t\ra\infty} \tau(t)\le \tau_0$ by observing that for any $\ve>0$ there exists $T_\ve$ with the property
\be \label{AB5}
\tau(t)\le\tau_0+\ve  \ {\rm for \ some \ }t\ge T_\ve \quad {\rm implies}  \ \tau(s)\le \tau_0+\ve \ {\rm for \ all \ } s\ge t. 
\ee
To see this note that if $\tau(s)=\tau_0+\ve$ then
\be \label{AC5}
\frac{d y(s)}{ds} \ > \ \frac{g(y(s))}{\beta(y(s),0)} \ ,
\ee
which implies $\tau(s')<\tau_0+\ve$ for $s'>s$ close to $s$. The inequality  (\ref{AB5}) follows from (\ref{B1}), (\ref{C4}) and  (\ref{M5})  on  choosing $T_\ve$ sufficiently large. Since we can see by a similar argument that  $\liminf_{t\ra\infty} \tau(t)\ge \tau_0$, we conclude that  $\lim_{t\ra\infty} \tau(t)= \tau_0$. It immediately follows from (\ref{B1}),  (\ref{C4}) and (\ref{M5})   that
\be \label{AD5}
\lim_{t\ra\infty} \tau(t)= \tau_0, \quad \lim_{t\ra\infty} \frac{1}{g(y(t))}\frac{dy(t)}{dt} \ = \ 1.
\ee
Hence we have from (\ref{C4}), (\ref{M5}) and (\ref{AD5}) that $\lim_{t\ra\infty}\langle X_t\rangle= e^{\al+\beta-\tau_0}<1$. 

To see that $\lim_{t\ra\infty} \kappa(t)=\kappa_0=\phi'(1)/\psi'(1)$,  we use the identity
\be \label{AE5}
\frac{d}{dt} \log u(t) \ = \ -\frac{G(u(t),v(t))+u(t)G_v(u(t),v(t))}{u(t)G_u(u(t),v(t))} \ ,
\ee
where $G(u,v)$ is the function (\ref{G3}). From (\ref{AD5}) we see that for any $\ve>0$ there exists $T_\ve>0$ such that if $t\ge T_\ve$ then 
\begin{eqnarray} \label{AF5}
- u(t)G_u(u(t),v(t)) \ &\ge& \ CG(u(t),v(t)), \\
 |G(u(t),v(t))+u(t)G_v(u(t),v(t))| \ &\le& \ \ve G(u(t),v(t)) \ , \nonumber
\end{eqnarray} 
where $C>0$ is independent of $\ve$.
The limit of the RHS of (\ref{AE5}) as $t\ra\infty$ is therefore  $0$, whence (\ref{C3}) implies $\lim_{t\ra\infty} \kappa(t)=\kappa_0$.

Finally we show that $\beta(\cdot,t)$ converges as $t\ra\infty$. The invariant solution $w_{\kappa_0}(\cdot)$ of (\ref{G1}) when $\kappa=\kappa_0$ is given by the formula
\be \label{AG5}
w_{\kappa_0}(x) \ = \ \exp\left[\tau_0-\beta-\frac{\al}{1-x}\right] \ , \quad  0\le x<1,
\ee
with $\tau_0,\al,\beta$ as in (\ref{Z5}).  Following the argument of Proposition 3.1 again, we define the function $g(x,t)$ by $w(x,t)=w_{\kappa_0}(x)g(x,t)$. From (\ref{M5}) and (\ref{AD5})  we see that  for any $\del$ with $0<\del<1$ there exists $T_\del>0$ such that
\be \label{AH5}
|(1-x)^2\frac{\pa}{\pa x} \log g(x,t)| \ \le \ \del \quad {\rm for \ } 0\le x\le 1-\del, \ t\ge T_\del \ .
\ee
Now (\ref{AH5}) implies that there is a constant $C$ independent of $\del$ such that
\be \label{AI5}
|c(x,t)-c_{\kappa_0}(x)g(x,t)| \ \le \ C \del  \ c_{\kappa_0}(x)g(x,t) \quad {\rm for \ } 0\le x\le 1-\del, \ t\ge T_\del \ .
\ee
We also have similarly to (\ref{AE3}) that for $0\le x\le 1-\del$ and $t\ge T_\del$, 
\begin{multline} \label{AJ5}
|h(x,t)-h_{\kappa_0}(x)g(x,t)| \ \le \ |h(1-\del,t)-h_{\kappa_0}(1-\del)g(1-\del,t)| \\
 +\int_x^{1-\del} h_{\kappa_0}(x')|\pa g(x',t)/\pa x'| \ dx' \ .
\end{multline}
From (\ref{AH5}) it follows that there is a constant $C$ independent of $\del$ such that
\be \label{AK5}
\int_x^{1-\del} h_{\kappa_0}(x')|\pa g(x',t)/\pa x'| \ dx' \  \le \ C\del h(x,t) \quad t\ge T_\del \ .
\ee
Consider any $\ve$ with $0<\ve<1$. It is clear that we may choose $\del<\ve$ and $T_\ve>0$ depending on $\ve$ such that 
\be \label{AL5}
w(1-\del,t) \ \ \le \ \ve w(1-\ve,t), \quad h(1-\del,t) \ \ \le \ \ve h(1-\ve,t) \  \quad  \ {\rm for \ } t\ge T_\ve . 
\ee
It follows from (\ref{AL5}) that there are constant $C,C'$ independent of $\ve$ such that
\begin{multline} \label{AU5}
 h_{\kappa_0}(1-\del)g(1-\del,t)\le C\del^2w_{\kappa_0}(1-\del)g(1-\del,t) \\
=C\del^2 w(1-\del,t) 
 \le \ C\ve\del^2 w(1-\ve,t)\le C\ve(1-x)^2 w(x,t)\\= C\ve (1-x)^2w_{\kappa_0}(x)g(x,t)
 \le C'\ve h_{\kappa_0}(x)g(x,t) \quad  \ {\rm for \ } 0\le x\le 1-\ve, \ t\ge T_\ve . 
\end{multline}
We conclude from (\ref{AI5})-(\ref{AU5}) that there is a constant $C$ independent of $\ve$ such that
\be \label{AM5}
|\beta(x,t)-\beta_{\kappa_0}(x)| \ \le \ C\ve \quad {\rm for \ } 0\le x\le 1-\ve, \ t\ge T_\ve \ .
\ee
\end{proof}
If we assume that (\ref{X1}) holds, then (\ref{AM5}) and the almost monotonicity of the function $\beta(\cdot,t)$ at large $t$ implies that 
\be \label{AN5}
\lim_{t\ra\infty}\|\beta(\cdot,t)-\beta_{\kappa_0}(\cdot)\|_\infty \ = \ 0 \ .
\ee
We give a direct proof of (\ref{AN5}) since it shows the key implication of the assumption (\ref{X1}) is that  it implies the function $g(\cdot)$ of (\ref{B4}) is monotonic decreasing. If $\log g(z)$ has large oscillations as $z\ra 1$ then (\ref{AN5}) may not hold.
\begin{proposition}
Suppose $\beta(\cdot,0)$ satisfies  (\ref{U1}) with $\beta_0=1$ and also (\ref{X1}).  Then (\ref{AN5}) holds.
\end{proposition}
\begin{proof}
We use the identity
\be \label{AO5}
\beta(x,t) \ = \ \beta(F(x,t),0) \int_{F(x,t)}^1 \frac{\pa F(x,t)/\pa x}{\pa F(x',t)/\pa x'} \ w_0(z) \ dz \Big/ \int_{F(x,t)}^1 w_0(z) \ dz \ ,
\ee
where $z=F(x',t), \ x\le x'<1$. Observe that for any $\del$ with $0<\del<1$ there is the inequality
\be \label{AP5}
\frac{\pa F(x,t)/\pa x}{\pa F(x',t)/\pa x'} \ \ge  \ (1-\del)^2 \ , \quad x\le x'\le x+\del(1-x) \ .
\ee
Hence it will be sufficient for us to show that there exists $\del_0,  \ve_0$ with $0<\del_0,\ve_0<1$ such that if $0<\del\le \del_0,  \ 0<\ve<\ve_0,$  then
\be \label{AQ5}
\limsup_{t\ra\infty}\sup_{1-x\le \ve}\frac{ h_0(F(x+\del(1-x),t))}{ h_0(F(x,t))} \ \le  \exp[-\al\del/2\ve] \ .
\ee
To prove (\ref{AQ5}) we use the identity $h_0(z) \ = \  g(z)w_0(z), \ 0\le z<1,$ where $g(\cdot)$ is the function (\ref{B4}). Since $g(\cdot)$ is decreasing, (\ref{AQ5}) follows from the same inequality with $h_0(\cdot)$ replaced by $w_0(\cdot)$.  We also have from (\ref{D4}) that
\be \label{AR5}
\frac{w_0(x+zg(x))}{w_0(x)}  \ \le \exp\left[ -z\inf_{x\le x'<1}{\beta(x',0)}\right] \quad {\rm   for \ } 0\le x\le x+zg(x)<1.
\ee

Observe now that
\be \label{AS5}
F(x+\del(1-x),t)-F(x,t) \ \ge \   \frac{\del u(t)}{u(t)+a(t)(1-x)} [1-F(x,t)] \ , \quad 0<x<1.
\ee
It follows then from (\ref{B4}), (\ref{AR5}), (\ref{AS5})  that for any $M>0$, 
\be \label{AT5}
\lim_{t\ra\infty} \sup_{1-x\le Mu(t)/v(t)} \frac{ w_0(F(x+\del(1-x),t))}{w_0(F(x,t))} \ = \ 0.
\ee
Since $\lim_{t\ra\infty} v(t)=\infty$, we also see that there exists constants  $T_0,M_0,C_1,C_2>0$  such that if  $t\ge T_0$ and $1-x\ge Mu(t)/v(t)$ for some $M\ge M_0$, then 
\be \label{AV5}
\left(\frac{\al}{1-x}+\beta\right)\left[1- \frac{C_1}{M}\right] \frac{dy(t)}{dt}\le F(x,t)-y(t)\le\left(\frac{\al}{1-x}+\beta\right)\left[1+ \frac{C_2}{M}\right] \frac{dy(t)}{dt} \ .
\ee
We conclude from (\ref{AD5}), (\ref{AR5}), (\ref{AV5}) that there exists $\del_0,  \ve_0$ with $0<\del_0,\ve_0<1$ such that if $0<\del\le \del_0,  \ 0<\ve<\ve_0,$ and $M\ge 1/\del^2$, then  
\be \label{AW5}
\limsup_{t\ra\infty} \sup_{ Mu(t)/v(t)\le 1-x\le \ve} \frac{ w_0(F(x+\del(1-x),t))}{w_0(F(x,t))} \ \le \ \exp[-\al\del/2\ve] \ .
\ee
The inequality (\ref{AQ5}) follows from (\ref{AT5}), (\ref{AW5}). 
\end{proof}

\vspace{.1in}

\section{Completion of the Proof of Theorem 1.3}
We  wish to formulate (\ref{A1}), (\ref{B1}) for general functions $\phi(\cdot),  \ \psi(\cdot)$ satisfying  (\ref{D1}), (\ref{E1})  in such a way that it can be approximated by the quadratic model studied in $\S3$ and $\S5$. In order to do this recall that the function $F(x,t)$ defined by (\ref{A2}) is the solution to the initial value problem (\ref{A3}), where the linear first order PDE contains a free parameter $\kappa(t), \ t\ge 0$. The conservation law (\ref{B1}) determines the function $\kappa(\cdot)$ uniquely, and in particular one sees that it  is strictly positive.  In (\ref{C3}) we defined a new parameter $u(t),  \ t\ge 0,$ in terms of $\kappa(\cdot)$, and it turned out that the dynamics of the quadratic model had the simple form (\ref{E3}) in terms of the function $u(\cdot)$.   We therefore formulate the general case in such a way that the free parameter is the function $u(\cdot)$ of (\ref{C3}) rather than the function $\kappa(\cdot)$ which enters in (\ref{A3}). 

To carry this out we write the characteristic equation (\ref{A2}) in terms of $u(\cdot)$.  Thus (\ref{A2}) is equivalent to
\be \label{A6}
\frac{dx(s)}{ds} \ = \ \phi(x(s))+\frac{\psi(x(s))}{\psi'(1)}\left[\frac{1}{u(s)}\frac{du(s)}{ds}-\phi'(1)\right] \ ,
\ee
whence we obtain the equation
\be \label{B6}
u(s)\frac{dx(s)}{ds} -\frac{\psi(x(s))}{\psi'(1)}\frac{du(s)}{ds} \ = \ u(s)\left[\frac{\psi'(1)\phi(x(s))-\psi(x(s))\phi'(1)}{\psi'(1)}\right] \  .
\ee
Next let $f(x), \ 0\le x< 1$, be the function defined by
\be \label{C6}
\frac{d}{dx}\log f(x) \ = \ -\frac{\psi'(1)}{\psi(x)} \ ,  \ 0\le x<1, \qquad \lim_{x\ra 1} (1-x)f(x) \ = \ 1.
\ee
 If the function $\psi(\cdot)$ is quadratic, it is easy to see from (\ref{C6}) that $f(\cdot)$ is given by the formula 
\be \label{D6}
f(x) \ = \  \frac{1}{1-x}-\frac{\psi''(1)}{2\psi'(1)} \ .
\ee
More generally $f:[0,1)\ra\R$ is a strictly increasing function satisfying $f(0)>0$ and  $\lim_{x\ra 1} f(x)=\infty$. Multiplying (\ref{B6}) by $f'(x(s))$, we conclude from (\ref{C6})  that
\be \label{E6}
\frac{d}{ds}\left[f(x(s))u(s)\right] \ = \ u(s)f'(x(s))\left[\frac{\psi'(1)\phi(x(s))-\psi(x(s))\phi'(1)}{\psi'(1)}\right] \  .
\ee

We define now the domains $\mathcal{D}=\{ (x,u)\in\R^2: 0<x<1, \ u>0\}$ and  $\hat{\mathcal{D}}=\{ (z,u)\in\R^2: z>f(0)u, \ u>0\}$. Then the transformation $(z,u)=(f(x)u, u)$ maps $\mathcal{D}$ to $\hat{\mathcal{D}}$.  Furthermore from (\ref{D6}) trajectories $x(s), \ s\le t,$ of (\ref{A2}) with $u(\cdot)$ defined in terms of the function $\kappa(\cdot)$ by (\ref{C3})  have the property that $(x(s),u(s))\in\mathcal{D}$ map under the transformation to $(z(s),u(s))\in\hat{\mathcal{D}}$, where $z(s)$ is a solution to
\be \label{F6}
\frac{dz(s)}{ds} \ = \ g(z(s),u(s)) \ , \quad s\le t, \ z(t)=z.
\ee
and $g(z,u)$ is the function
\be \label{G6}
g(z,u) \ = \  uf'(x)\left[\frac{\psi'(1)\phi(x)-\psi(x)\phi'(1)}{\psi'(1)}\right] \  .
\ee
\begin{lem}
Assume $\phi(\cdot), \ \psi(\cdot)$ satisfy (\ref{D1}), (\ref{E1}). Then there are positive constants $C_1,C_2$ such that  $-C_2u\le g(z,u)\le -C_1 u$  for $(z,u)\in\hat{\mathcal{D}}$ and
\be \label{H6}
\lim_{z\ra\infty} g(z,u) \ = \  \frac{u[\psi'(1)\phi''(1)-\psi''(1)\phi'(1)]}{2\psi'(1)} \ = \ -\al_0 u \ ,
\ee
where $\al_0>0$.  The function $z\ra g(z,u)$ is  $C^2$ in the interval $z>f(0)u$ and $\pa g(z,u)/\pa z$ is given by the formula
\be \label{AA6}
\frac{\pa g(z,u)}{\pa z} \ = \  \Ga(x) \ = \ \phi'(x)+\phi'(1)-\phi(x)[\psi'(x)+\psi'(1)]/\psi(x) , \quad z=f(x)u,
\ee
where $\Ga(\cdot)$ is $C^1$ on the interval $(0,1]$ and satisfies $\Ga(1)=\Ga'(1)=0$. 

If in addition $\phi(\cdot), \ \psi(\cdot)$ satisfy (\ref{O1}) then $\Ga(\cdot)$ is $C^2$ on $(0,1]$ and  $g(z,u)$ is an increasing function of $z>f(0)u$. The function $z\ra g(z,u)$ is concave for $z>f(0)u$ 
provided $\phi(\cdot), \ \psi(\cdot)$ satisfy (\ref{Y1}). The condition (\ref{Y1}) holds if  $\phi(\cdot), \ \psi(\cdot)$ satisfy (\ref{D1}), (\ref{E1}), (\ref{O1})  and $\psi(\cdot)$ is quadratic. 
 \end{lem} 
\begin{proof}
From (\ref{D1}), (\ref{E1}) we see that the function $h(x)= \psi'(1)\phi(x)-\psi(x)\phi'(1)$ is convex and satisfies $h(0)>0,  \ h(1)=0, \ h'(1)=0,$ whence $h(\cdot)$ is decreasing and strictly positive for $0\le x<1$. It follows that if $0<\del\le 1$ there are positive constants $C_{1,\del}, \ C_{2,\del}$ such that  $-C_{2,\del}u\le g(z,u)<-C_{1,\del}u$ for $ f(0)u\le z\le f(1-\del)u$.  Observe further from  (\ref{G6}) that we  may write the function $g(z,u)$ as
\be \label{I6}
g(z,u) \ = \ \frac{uf'(x)(1-x)^2}{2\psi'(1)} \int_0^1[\psi'(1)\phi''(\la x+1-\la )-\psi''(\la x+1-\la)\phi'(1)] \ d\rho(\la) \ ,
\ee
where $\rho(\cdot)$ is a probability measure on the interval $[0,1]$. Since $\psi''(1)-\phi''(1)>0$  and $\lim_{x\ra 1} f'(x)(1-x)^2=1,$   it follows from (\ref{I6}) that we may choose $C_{1,\del}, \ C_{2,\del}$ independent of $\del$ as $\del\ra 0$.   Evidently (\ref{I6}) implies (\ref{H6})  on using the fact that $\lim_{x\ra 1} f'(x)(1-x)^2=1$. 

To see that $g(z,u)$ is an increasing function of $z> f(0)u$,  we show that $\pa g(z,u)/\pa x\ge 0$ for $0\le x<1$.  From (\ref{C6}), (\ref{G6}) we have that
\be \label{J6}
\frac{\pa g(z,u)}{\pa x} \ = \ u\left\{\psi(x)[\phi'(x)+\phi'(1)]-\phi(x)[\psi'(x)+\psi'(1)]\right\}|\psi'(1)|f(x)/\psi(x)^2 \ .
\ee
Consider now the function $k(x)=(1-x)[\phi'(x)+\phi'(1)]+2\phi(x)$, which has the property that $k(1)=k'(1)=0$ and $k''(x)=(1-x)\phi'''(x)$.   Assuming $\phi''(\cdot)$ is increasing, it follows that $k(\cdot)$ is convex and hence non-negative for $0\le x<1$.  Since we can make a similar argument for $\psi(\cdot)$ under the assumption that $ \psi''(\cdot)$ is decreasing, we obtain the inequalities
\be \label{K6}
\phi'(x)+\phi'(1) \ \ge \ -2\phi(x)/(1-x), \quad \psi'(x)+\psi'(1) \ \le \ -2\psi(x)/(1-x), \ \ 0\le x<1.
\ee
Now (\ref{J6}), (\ref{K6}) imply that  $\pa g(z,u)/\pa z\ge 0$  for $z\ge f(0)u$.  The formula (\ref{AA6}) follows from (\ref{C6}) and (\ref{J6}).  Hence the function $z\ra g(z,u)$ is concave if $\Ga(x)$ is a decreasing function of $x$. 

If $\psi(\cdot)$ is quadratic then (\ref{AA6}) implies that
\be \label{AB6}
\frac{\pa g(z,u)}{\pa z} \ = \ \phi'(x)+\phi'(1)+2\phi(x)/(1-x) \ ,
\ee
and so
\be \label{AC6}
\frac{\pa}{\pa x} \frac{\pa g(z,u)}{\pa z} \ = \ \phi''(x)+2\phi'(x)/(1-x)+2\phi(x)/(1-x)^2  \ .
\ee
Now just as before the condition $\phi'''(\cdot)\ge 0$ implies that the RHS of (\ref{AC6}) is not positive for $0< x<1$.
\end{proof}
\begin{rem}
Observe that we have in the case of quadratic $\phi(\cdot)$, for example $\phi(x)=x(1-x)$, the identity
\be \label{AH6}
\Ga'(0) \ = \ \phi''(0){\color{red}-}\phi'(0)[\psi'(0)+\psi'(1)]/\psi(0) \ = \  -2{\color{red}-}[\psi'(0)+\psi'(1)]/\psi(0)  \ .
\ee
We have already seen in (\ref{K6})  that if (\ref{O1}) holds then the RHS of (\ref{AH6}) is positive if $\psi(\cdot)$ is not quadratic. Hence the function $z\ra g(z,u)$ is concave when $\phi(\cdot)$ is quadratic only if $\psi(\cdot)$ is also quadratic. 
\end{rem}

Observe that the condition $\kappa(\cdot)$ a positive function, which ensures that trajectories $(x(s),u(s)),  \ s\le t, $ of (\ref{A2}) with $(x(t),u(t))\in\mathcal{D}$ remain in $\mathcal{D}$, becomes the condition 
\be \label{L6}
\frac{d}{dt} \log u(t) \ > \ \phi'(1) \ , \quad t\ge 0.
\ee
Hence if $u(\cdot)$ satisfies (\ref{L6}) then solutions $(z(s),u(s)), \ s\le t,$ of (\ref{F6}) with $(z(t),u(t))\in\hat{\mathcal{D}}$ remain in $\hat{\mathcal{D}}$. To see this directly  first observe from (\ref{C6}), (\ref{G6})  that $g(f(0)u,u)=\phi'(1)f(0)u<0$.  For the trajectory $(z(s),u(s)), \ s\le t,$ to remain in $\hat{\mathcal{D}}$ we must have
\be \label{M6}
\frac{dz}{du} \ > \ f(0) \quad {\rm if \ } (z(s),u(s))\in\pa\hat{\mathcal{D}} \quad {\rm and \ } \frac{du(s)}{ds}<0,
\ee
since $dz(s)/ds<0$. 
Now (\ref{G6}) implies in this case that
\be \label{N6}
\frac{dz}{du} \ = \ g(z(s),u(s))\Big/  \frac{du(s)}{ds} \ = \ \phi'(1)f(0)u(s)\Big/  \frac{du(s)}{ds} \ > \ f(0) \ .
\ee

The first order PDE with characteristic equation (\ref{F6})  is given by
\begin{eqnarray} \label{O6}
\frac{\pa \hat{F}(z,t)}{\pa t}+ g(z,u(t))\frac{\pa \hat{F}(z,t)}{\pa z}&=&0, \quad z>f(0)u(t), \ t\ge 0,
\\
\hat{F}(z,0) &=& z, \quad z>f(0). \nonumber
\end{eqnarray}
Comparing now (\ref{A3}) to (\ref{O6}) and using (\ref{E6}), we conclude that the solutions $F(x,t)$ of (\ref{A3}) and $\hat{F}(z,t)$ of (\ref{O6}) are related by the identity
\be \label{P6}
f(F(x,t)) \ = \ \hat{F}( f(x)u(t),t) \ , \quad 0<x<1,  \  t>0. 
\ee
In the case when $\psi(\cdot), \ \phi(\cdot)$ are quadratic functions, the solution to (\ref{O6}) is given by the formula
\be \label{Q6}
\hat{F}(z,t) \ = \ z+\al_0v(t) \ ,
\ee
where $v(t), \ t\ge 0,$ is the solution to  (\ref{D3}). We easily conclude from (\ref{D6}), (\ref{O6}), (\ref{P6}) that in the quadratic case  $F(x,t)$ is given by the formula (\ref{E3}). More generally we have as a consequence of Lemma 6.1  the following:
\begin{corollary}
Assume $\phi(\cdot), \ \psi(\cdot)$ satisfy (\ref{D1}), (\ref{E1}). Then for $t\ge 0$ the function $z\ra\hat{F}(z,t)$ with domain $\{z\ge f(0)u(t)\}$ is increasing, and there are positive constants $C_1,C_2$ such that  $z+C_1v(t)\le \hat{F}(z,t)\le z+ C_2v(t)$. If in addition (\ref{O1}) holds  then  $\pa\hat{F}(z,t)/\pa z\le 1$. If the function $z\ra g(z,u)$ is concave for all $u>0$ then $\hat{F}(z,t)$ is a convex function of $z>f(0)u(t)$. 
\end{corollary}
\begin{proof}
We have that $\hat{F}(z,t)=z(0)$, where $z(s),  \ s\le t,$ is the solution to (\ref{F6}) with $z(t)=z$, whence it follows that the function  $z\ra\hat{F}(z,t)$ is increasing. From Lemma 6.1 it follows that
\be \label{R6}
z(s) \ \le z+C_2\int^t_s u(s') \ ds' \ , \quad 0\le s\le t \ ,
\ee
and so $\hat{F}(z,t)\le z+ C_2v(t)$.  We conclude that  $\pa\hat{F}(z,t)/\pa z\le 1$ from the formula
\be \label{S6}
\frac{\pa\hat{F}(z,t)}{\pa z}\ = \  \exp\left[-\int_0^t \frac{\pa g(z(s),u(s))}{\pa z} \ ds \ \right] \ 
\ee
and Lemma 6.1. Evidently (\ref{S6}) implies the convexity of the function $z\ra\hat{F}(z,t)$ is a consequence of the concavity of the function $z\ra g(z,u)$.  
\end{proof}
Next we show that $\limsup_{t\ra\infty} v(t)/u(t)=\infty$ if $\lim_{x\ra 1}\beta(x,0)=1$.  We can already obtain from the results of $\S4$ a positive lower bound  $\liminf_{t\ra\infty}v(t)/u(t)>0$. To see this note that we have shown that  $\sup\kappa(\cdot)\le M<\infty$ and hence  (\ref{C3}) implies  that
\be \label{T6}
u(s) \ \ge \ \frac{1}{M|\psi'(1)|}\frac{du(s)}{ds} \ , \quad s\ge 0.
\ee
 We conclude that $v(t)\ge [u(t)-1]/M|\psi'(1)|$ for $t\ge 0$. Since (\ref{D3}) also implies that $v(t)\ge M_1>0$ for all $t\ge 1$, it follows that  there exists $M_2>0$ such that  $v(t)\ge M_2 u(t)$ for $t\ge 1$. 
\begin{corollary}
Assume $\phi(\cdot), \ \psi(\cdot)$ satisfy (\ref{D1}), (\ref{E1}) and that $\lim_{x\ra1}\beta(x,0)=1$. Then if $u(\cdot), \ v(\cdot),$ are given by (\ref{C3}), (\ref{D3})  one has  $\liminf_{t\ra\infty}v(t)/u(t)>0$ and $\limsup_{t\ra\infty} v(t)/u(t)=\infty$.  
\end{corollary} 
\begin{proof}
Now $w(x,t)=e^tw(F(x,t),0)$, whence it follows   from (\ref{B1})  that
\be \label{AD6}
w(F(0,t),0)\ge e^{-t} \ , \quad w(F(1/2,t),0)\le 2e^{-t} \ .
\ee
We also have from (\ref{P6}) and Corollary 6.1  that
\be \label{U6}
f(x)u(t)+C_1 v(t) \ \le \ f(F(x,t)) \le \ f(x)u(t)+C_2 v(t) \ .
\ee
Since $\lim_{x\ra 1} f(x)(1-x)=1$, we conclude from (\ref{AD6}), (\ref{U6}) and Lemma 2.1 that there are positive constants $T_0,C_3,C_4$ such that 
\be \label{AE6}
w\left(1-\frac{C_3}{u(t)+v(t)},0\right)\ge e^{-t} \  {\rm and \ } w\left(1-\frac{C_4}{u(t)+v(t)},0\right)\le 2e^{-t} \ {\rm for \ } t\ge T_0\ .
\ee
Hence if $z(t)$ is defined as in Lemma 4.2 by $w(z(t),0)=e^{-t}$,   then (\ref{AE6}) implies that 
\be \label{AF6}
\frac{C_4}{1-z(t-\log2)} \ \le \ u(t)+v(t) \ \le \  \frac{C_3}{1-z(t)} \quad  {\rm for \ } t\ge T_0 \ .
\ee
Observe from (\ref{D3}) that
\be \label{CA6}
v(t) \  = \  \int_0^t e^{s-t}[u(s)+v(s)] \ ds \ ,
\ee
 and so we conclude from  (\ref{AF6}), (\ref{CA6})  that
\be \label{AG6}
v(t) \ \ge \  C_4\int_{t-1}^t e^{s-t}\frac{ds}{1-z(s-\log2)}  \quad  {\rm for \ } t\ge T_0+1 \ .
\ee 
Observe next from (\ref{C4})  and the fact that $\lim_{x\ra 1}g(x)/(1-x)=0$, that we can choose $T_0$ sufficiently large so that $1-z(t-1-\log2)\le 2[1-z(t)]$ for $t\ge T_0+1$. We conclude from (\ref{AF6}), (\ref{AG6}) that $v(t)/u(t)\ge C_4(e-1)/2C_3 e$ provided $t\ge T_0+1$. 

To prove that $\limsup_{t\ra\infty} v(t)/u(t)=\infty$ we assume for contradiction that there is a constant $K$ such that $v(t)\le Ku(t)$ for $t\ge 0$. 
By Lemma 2.1 and (\ref{U6}) we see that $\lim_{t\ra\infty}[u(t)+v(t)]=\infty$, and so we conclude that $\lim_{t\ra\infty} u(t)=\infty$.  We also see from (\ref{U6}) that
\be \label{V6}
f(x)u(t) \ \le \ f(F(x,t)) \le \ [f(x)+C_2K] u(t)  \ , \quad t\ge 0.
\ee
We define functions $G_1(u), \ G_2(u)$ with domain $u\ge 1$ by
\be \label{W6}
G_1(u)  =  \int_0^1 w(f^{-1}[f(x)u],0) \ dx  , \quad  G_2(u)  =  \int_0^1 w(f^{-1}[\{f(x)+C_2K\}u],0) \ dx.
\ee
 Evidently $G_1(\cdot), \ G_2(\cdot)$ are strictly decreasing functions satisfying $\lim_{u\ra\infty} G_j(u)=0, \ j=1,2$ and $G_1(u)\ge G_2(u)$ for all $u\ge 1$. Hence there exists $T_0\ge 0$ such that there are  strictly increasing functions $u_j(t), \ j=1,2$ with domain $t\ge T_0$ such that $G_j(u_j(t))=e^{-t}, \ j=1,2$.   
 It follows from (\ref{V6}) that $u_2(t)\le u(t)\le u_1(t)$ for $t\ge T_0$, and hence
 \be \label{X6}
 \frac{v(t)}{u(t)} \ \ge \ \frac{1}{u_1(t)}\int_{T_0}^t u_2(s) \ ds, \quad t\ge T_0\  .
 \ee
 We obtain a contradiction to the assumption $\sup[v(\cdot)/u(\cdot)]\le K$  by showing that the RHS of (\ref{X6}) converges to $\infty$ as $t\ra\infty$. 
 
To see this let $\eta=\inf_{0\le x<1}[f(x)/\{f(x)+C_2K\}]$ so $0<\eta<1$ and $u_2(t)\ge \eta u_1(t)$ for $t\ge T_0$. Observe next that there is a positive constant $C_3$ such that
\begin{multline} \label{Y6}
f^{-1}[\{f(x)+C_2K\}2u]-f^{-1}[\{f(x)+C_2K\}u] \ \ge \\
 C_3\{1- f^{-1}[\{f(x)+C_2K\}u] \ \} \ , \quad 0\le x<1, \ u\ge 1.
\end{multline}
Since the function $g(\cdot)$ of (\ref{B4}) satisfies $\lim_{x\ra 1} g(x)/(1-x)=0$, it follows from (\ref{C4}), (\ref{Y6}), that $\lim_{u\ra\infty} G_2(2u)/G_2(u)=0$.  Hence for any $\del>0$ there exists $u_\del\ge 1$ such that $G_2(2u)/G_2(u)\le\del$ for $u\ge u_\del$. Since $\lim_{t\ra\infty} u_1(t)=\infty$ and $\liminf_{t\ra\infty}u_2(t)/u_1(t)>0$,  it also follow that $\lim_{t\ra\infty} u_2(t)=\infty$. Hence there exists $T_\del$ such that $u_2(t)\ge u_\del$ for all $t\ge T_\del$.  It follows that if $t_0\ge T_\del$ then 
\be \label{Z6}
 t_0\le t\le t_0+\log(1/\del) \quad {\rm implies \ } u_2(t)\ge u_2(t_0+\log(1/\del))/2 \ .
\ee
We conclude that the RHS of (\ref{X6}) is bounded below by $\eta\log(1/\del)/2$ provided $t\ge T_\del+\log(1/\del)$. 
\end{proof}
In order to prove the inequality (\ref{AM4}) and obtain a lower bound on $\kappa(\cdot)$ in the case when $\lim_{x\ra 0} \phi(x)/x<\infty$,  we need to consider the dependence of the function $u(t), \ t\ge 0,$ on $v(t), \  t\ge 0$. Since $v(t)$ is a strictly increasing function of $t$ we may write $u(t)=U(v(t)), \ t\ge 0$. It follows from (\ref{C3}), (\ref{D3}), Lemma 4.1 and Corollary 6.2 that
\be \label{AI6}
\phi'(1) \ \le \ U'(v) \ \le \  C \ {\rm for \ } v\ge 0, \quad U(v) \ \le \ Cv  \ {\rm for \ } v\ge 1,
\ee
where $C$ is a positive constant. 
\begin{lem}
Assume $\phi(\cdot), \ \psi(\cdot)$ satisfy (\ref{D1}), (\ref{E1}), (\ref{O1}), (\ref{Y1}) and either that $\lim_{x\ra 0}\phi(x)/x=\infty$ or the functions $\phi(\cdot),  \psi(\cdot)$ are $C^2$ on the closed interval $[0,1]$.  Assume also that the solution $w(x,t)$ of (\ref{A1}), (\ref{B1}) satisfies (\ref{U1}) with $0<\beta_0\le 1$. Then for any $k\ge 1$ there is a constant $C_k$, independent of $t\ge 0$, which is an  increasing function of $k$ and satisfying $\lim_{k\ra 1} C_k=1,$ such that
\be \label{AJ6}
\frac{\pa\hat{F}(f(0)u(t),t)}{\pa z} \ \le \ \frac{\pa\hat{F}(kf(0)u(t),t)}{\pa z} \ \le \ C_k\frac{\pa\hat{F}(f(0)u(t),t)}{\pa z} \ , \quad t\ge 0.
\ee
\end{lem} 
\begin{proof}
Let $z_k(s), \  s\le t,$ be the solution to (\ref{F6}) with terminal condition $z_k(t)=kf(0)u(t)$. From the convexity of $\hat{F}(\cdot,t)$  and (\ref{S6}) it will be sufficient for us to show that
\be \label{AK6}
\int_0^t \frac{\pa g(z_k(s),u(s))}{\pa z} \ ds \ \ge \   \int_0^t \frac{\pa g(z_1(s),u(s))}{\pa z} \ ds -D_k, \quad k\ge 1,
\ee
for a constant $D_k$ depending only on $k$ which satisfies $\lim_{k\ra 1} D_k=0$.  Making the change of variable $t\leftrightarrow v, \ s\leftrightarrow v'$, we have that the integral in (\ref{AK6}) can be written as
\be \label{AL6}
\int_0^t \frac{\pa g(z_k(s),u(s))}{\pa z} \ ds \ = \ \int_0^v \frac{\pa g(\tilde{z}_k(v'),U(v'))}{\pa z} \ \frac{dv'}{U(v')} \ ,
\ee
where $z_k(s)=\tilde{z}_k(v')$. 
Observe now  that
\be \label{AM6}
\tilde{z}_k(v') \ \ge  \ kf(0)U(v)+C_1\{v-v'\} \quad {\rm for \ } 0\le v'\le v,
\ee
where $C_1$ is the constant in Lemma 6.1.  Upon using the properties of the function $\Ga(\cdot)$ stated in Lemma 6.1, it also follows  from (\ref{AM6}) and the  second inequality of (\ref{AI6}) that  there are positive constants $C,\ga$ with $0<\ga\le 1$ such that 
\be \label{AN6}
 \frac{\pa g(\tilde{z}_k(v'),U(v'))}{\pa z}  \ \le \ \frac{CU(v')^2}{[f(0)U(v)+C_1\{v-v'\}]^2}  \quad {\rm for \ } 0\le v'\le \ga v,  \ v\ge 1, \ k\ge 1.
\ee
In  the case when the functions $\phi(\cdot),\psi(\cdot)$ are $C^1$ on the closed interval $[0,1]$ we can take $\ga=1$ in (\ref{AN6}). Otherwise we need to take $\ga<1$. 
We conclude that  there is a constant $C$ such that
\be \label{AO6}
\int_0^{v\min\{\ga,1/2\}} \frac{\pa g(\tilde{z}_k(v'),U(v'))}{\pa z} \ \frac{dv'}{U(v')} 
 \ \le  \ C \quad {\rm  for \ } v\ge 1, \ k\ge 1.
\ee
We also have from the properties of the function $\Ga(\cdot)$ that  if $0<\del<1$ then there is a constant $C_\del$ such that
\be \label{AP6}
0 \ \le \ \frac{\pa g(z,u)}{\pa z}-\frac{\pa g(z',u)}{\pa z'} \ \le \ \frac{C_\del u^2(z'-z)}{z^3} \quad {\rm for \ } 
f(\del)u\le z\le z' \ .
\ee
Observe now from Corollary 6.1 that
\be \label{AQ6}
\tilde{z}_1(v') \ \le \ \tilde{z}_k(v') \ \le \ \tilde{z}_1(v')+(k-1)f(0)U(v) \quad {\rm for \ } 0\le v'\le v.
\ee
It follows then from (\ref{AP6}), (\ref{AQ6})  that there is a constant $C$ such that
\be \label{AR6}
\int_0^{v\min\{\ga,1/2\}} \left[\frac{\pa g(\tilde{z}_1(v'),U(v'))}{\pa z} -\frac{\pa g(\tilde{z}_k(v'),U(v'))}{\pa z}\right]\ \frac{dv'}{U(v')} 
 \ \le  \ C(k-1) \quad {\rm  for \ } v\ge 1, \ k\ge 1.
\ee

Next we note that there exists $\del_0>0$ such that if $0<\del\le \del_0$ then there is a constant $\eta(\del)$ with the property  $\lim_{\del\ra0}\eta(\del)=0$ such that
\be \label{AS6}
\int_{v-\del U(v)}^v \frac{\pa g(\tilde{z}_1(v'),U(v'))}{\pa z} \ \frac{dv'}{U(v')}   \ \le \  \eta(\del) \quad {\rm if \ } v\ge 1.
\ee
 The inequality (\ref{AS6}) follows from (\ref{AI6}) in the case when the functions $\phi(\cdot),\psi(\cdot)$ are $C^1$ on the closed interval $[0,1]$ since then we can take $\ga=1$ in the inequality (\ref{AN6}).  In the case when $\lim_{x\ra\ 0}\phi(x)/x=\infty$ we need to use  Corollary 4.1 that $\inf\kappa(\cdot)>0$.   Defining $T$ by $v(T)=v$ we have from (\ref{C3}), (\ref{D3})  that
 \be \label{AT6}
 v(T-t_1) \ \le \ v(T)-t_1 e^{C_2\psi'(1)t_1} u(T) \quad {\rm for \ } 0\le t_1\le T, 
 \ee
where $C_2=\sup\kappa(\cdot)$.  Hence there exists $\del_1>0$ such that  for $0<\del\le \del_1$ one has
\be \label{AU6}
\int_{v-\del U(v)}^v \frac{\pa g(\tilde{z}_1(v'),U(v'))}{\pa z} \ \frac{dv'}{U(v')}   \ \le \ 
\int_{T-2\del}^T \Ga(x(s)) \ ds \ .
\ee
We conclude  from (\ref{A2}), (\ref{AA6}) upon using the inequality $\inf\kappa(\cdot)>0$ that
\be \label{AV6}
\int_{T-2\del}^T \Ga(x(s)) \ ds \ \le \ \eta(\del) \quad {\rm where \ } \lim_{\del\ra 0}\eta(\del)=0.
\ee

It follows from (\ref{AQ6}) and Lemma 6.1 that
\be \label{AW6}
\tilde{z}_k(v') \   \le  \ \tilde{z}_1(v'-(k-1)f(0)U(v)/C_1) \quad {\rm if \ } v'\ge(k-1)f(0)U(v)/C_1 ,
\ee
where $C_1$ is the constant of Lemma 6.1. Hence there exists $k_0>1, \ \del_2>0$ such that for $v\ge 1, \ 1\le k\le k_0, \ 0\le\del\le\del_2$, one has
\be \label{AX6}
\int_{v\min\{\ga,1/2\}}^{v-\del U(v)} \frac{\pa g(\tilde{z}_k(v'),U(v'))}{\pa z} \ \frac{dv'}{U(v')} 
\ge \ \int_{v\min\{\ga,1/2\}}^{v-\del U(v)-\rho} \frac{\pa g(\tilde{z}_1(v'),U(v'+\rho))}{\pa z} \ \frac{dv'}{U(v'+\rho)} \ ,
\ee
where $\rho=(k-1)f(0)U(v)/C_1$. Next observe that
\be \label{AY6}
 \frac{\pa g(z,U_1)}{\pa z}\frac{1}{U_1}-\frac{\pa g(z,U_2)}{\pa z}\frac{1}{U_2} \ = \ [f(x_1)\Ga(x_1)-f(x_2)\Ga(x_2)]/z \ , \quad {\rm where \ } f(x_1)U_1=z,  \ f(x_2)U_2=z.
\ee
We see now from (\ref{C6}), (\ref{AY6}) and Lemma 6.1 that for any $\ve$ satisfying $0<\ve\le 1$ there is a constant $C_\ve$ depending on $\ve$ such that 
\be \label{AZ6}
\left|  \frac{\pa g(z,U_1)}{\pa z}\frac{1}{U_1}-\frac{\pa g(z,U_2)}{\pa z}\frac{1}{U_2}\right| \ \le \ 
\frac{C_\ve|U_1-U_2|}{z^2} \quad {\rm for \ } \ve\le x_1,x_2\le 1 \ .
\ee
If the functions $\phi(\cdot),\psi(\cdot)$ are $C^2$ on the closed interval $[0,1]$ then $\lim_{\ve\ra 0}C_\ve=C_0<\infty$, but in the case $\lim_{x\ra 0}\phi(x)/x=\infty$  it is possible that $C_\ve$ becomes unbounded as $\ve\ra 0$.  
To estimate from below the integral on the RHS of (\ref{AX6}) we take $z=\tilde{z}_1(v'), \ U_1=U(v'+\rho), \ U_2=U(v')$ in (\ref{AZ6}) with $v\min\{\ga,1/2\}\le v'\le v-\del U(v)-\rho$. Since $\inf\kappa(\cdot)>0$ if $\lim_{x\ra 0}\phi(x)/x=\infty$ we may take $\ve=\ve(\del)>0$ in (\ref{AZ6}) in that case. 
We conclude then from (\ref{AI6}), (\ref{AM6}), (\ref{AZ6}) that
\be \label{BA6}
\int_{v\min\{\ga,1/2\}}^{v-\del U(v)-\rho} \left|\frac{\pa g(\tilde{z}_1(v'),U(v'+\rho))}{\pa z} \ \frac{1}{U(v'+\rho)}-\frac{\pa g(\tilde{z}_1(v'),U(v'))}{\pa z} \ \frac{1}{U(v')}\right| \ dv' \ \le \ C_\del (k-1) \ ,
\ee
where $C_\del$ depends only on $\del$ and can diverge as $\del\ra 0$ in the case when $\lim_{x\ra 0}\phi(x)/x=\infty$. It follows now from (\ref{AR6}), (\ref{AS6}), (\ref{BA6}) that there exists $k_0>1$ such that  (\ref{AK6}) holds for $1\le k\le k_0$. To prove the result for $k\ge k_0$ we repeat the argument but in this case we do not need to be concerned with the case $\lim_{x\ra 0}\phi(x)/x=\infty$. 
\end{proof}
\begin{corollary}
Assume that $\phi(\cdot), \ \psi(\cdot)$  and the solution $w(x,t)$ of (\ref{A1}), (\ref{B1}) satisfy the assumptions of Lemma 6.2. Then for any $\ve$ with $0<\ve<1$ there is a constant $\ga_\ve$ such that the function $F(x,t)$ defined by (\ref{A2}) has the property
\be \label{BB6}
\frac{\pa F(0,t)}{\pa x} \ \le \ \frac{\pa F(x,t)}{\pa x} \ \le [1+\ga_\ve]\frac{\pa F(0,t)}{\pa x} \quad {\rm for \ } 0\le x\le \ve, \ t\ge 0, 
\ee
and $\lim_{\ve\ra0} \ga_\ve=0$.
\end{corollary}
\begin{proof}
From (\ref{P6}) we have the identity
\be \label{BC6}
\frac{\pa F(x,t)}{\pa x} \ = \ \frac{f'(x)u(t)}{f'(F(x,t))} \frac{\pa \hat{F}(f(x)u(t),t)}{\pa z} \  .
\ee
From (\ref{C6}) we see that $f'(\cdot)$ is an increasing function, and since $F(\cdot,t)$ is also increasing we conclude that $f'(F(0,t))\le f'(F(x,t))$ for $0\le x<1$.  The result follows from (\ref{BC6}) and Lemma 6.2.
\end{proof}
\begin{proposition}
Assume that $\phi(\cdot), \ \psi(\cdot)$  and the solution $w(x,t)$ of (\ref{A1}), (\ref{B1}) satisfy the assumptions of Lemma 6.2. If $\lim_{x\ra 0}\phi(x)/x=\infty$ then (\ref{AM4}) holds. If $\phi(\cdot), \ \psi(\cdot)$ are $C^2$ on the closed interval $[0,1]$ and the initial data additionally  satisfies (\ref{X1}), then $\inf\kappa(\cdot)>0$ and (\ref{AM4}) holds. 
\end{proposition} 
\begin{proof}
Assuming first that $\inf\kappa(\cdot)>0$, we see from (\ref{B1}) and Lemma 2.3 that there exists $\al>0$ such that $1\le w(0,t)\le 1+\al$ for all $t\ge 0$. We conclude from (\ref{B1})  that
\be \label{BD6}
\frac{w_0(F(1/2(1+\al),t))}{w_0(F(0,t))} \ = \ \frac{w(1/2(1+\al),t)}{w(0,t)} \ \ge \ \frac{1}{1+2\al} \ . 
\ee
In view of the convexity of the function $F(\cdot,t)$ it follows from (\ref{C4}) and (\ref{BD6}) that there exists $T_0>0$ such that
\be \label{BE6}
\frac{\pa F(0,t)}{\pa x} \ \le \  3(1+\al)\log(1+2\al) g(F(0,t)) \quad {\rm for \ } t\ge T_0 \ .  
\ee
It follows from Corollary 6.3 and (\ref{BE6}) that if $t\ge T_0$ then
\begin{multline} \label{BF6}
w(x,t) \ = \ e^tw_0(F(x,t)) \  \ge \ e^tw_0(F(0,t)+x\pa F(x,t)/\pa x) \ \ge \\
e^tw_0(F(0,t)+x(1+\ga_x)\pa F(0,t)/\pa x) \ \ge \  C(x,\al) e^tw_0(F(0,t)) \ \ge \ C(x,\al) ,
\end{multline}
for a positive constant $C(x,\al)$ depending only on $x,\al$. We have proved the inequality (\ref{AM4}).

Finally we need to show that $\inf\kappa(\cdot)>0$ in the case when the functions $\phi(\cdot), \ \psi(\cdot)$ are $C^2$ on the closed interval $[0,1]$ and the initial data additionally  satisfies (\ref{X1}). We show that for any $\del>0$ there exists $T_\del, K_\del>0$ such that if $t\ge T_\del$ and $w(0,t)\ge K_\del>2$, then $\beta(0,t)\ge 1-\del$.  To see this  let us suppose that $w(0,t)=e^tw_0(F(0,t))\ge K_\del>2$, whence it follows from (\ref{B1}), (\ref{C4}) and (\ref{BB6}) that there are positive constants $T_0, C_0$ such that
\be \label{BG6}
\frac{\pa F(0,t)}{\pa x}  \ \ge \  C_0g(F(0,t))\log K_\del \quad {\rm for \ } t\ge T_0 \ .
\ee
We conclude from (\ref{C4}), (\ref{BB6}), (\ref{BG6}) that there is a constant $C_1>0$ such that
\be \label{BH6}
\frac{w_0(F(x,t))}{w_0(F(0,t))} \ \le \  \exp\left[-C_1x\log K_\del\right] \quad {\rm for \ } 0\le x\le 1/2.
\ee
Now just as in Proposition 5.1 we see  that the ratio $h_0(F(x,t))/h(F(0,t))$ is also bounded by the RHS of (\ref{BH6}) since we are assuming that the initial data satisfies (\ref{X1}).  Thus from (\ref{AO5}) we obtain for any $\ve\le 1/2$  the lower bound
\be \label{BI6}
\beta(0,t) \ \ge \ \beta(F(0,t),0)\left\{1- \exp\left[-C_1\ve\log K_\del\right]\right\}/(1+\gamma_\ve) \quad {\rm for \ }t\ge T_0 \ .
\ee
It is clear from (\ref{BI6}) that we may choose $K_\del,T_\del$ such that $\beta(0,t)\ge 1-\del$  if $t\ge T_\del$ and $w(0,t)\ge K_\del$. 

To complete the proof of $\inf\kappa(\cdot)>0$ we argue as in Corollary 4.1. Thus
\be \label{BJ6}
\frac{d}{dt}\log   \ \langle X_t\rangle \ \ge (1-\del)\left[\frac{ \langle\phi(X_t)\rangle}{\langle X_t\rangle}+1\right]-1 \ \quad {\rm if \ }  \langle X_t\rangle \le 1/K_\del \ , t\ge T_\del \ .
\ee
We have already observed in Corollary 4.1 that there exists $\ga>0$ such that $P(X_t>\ga\langle  X_t\rangle)>1/2$ for $t\ge 0$. Let $x_0\in(0,1)$ be the point at which the function $\phi(\cdot)$ achieves its maximum.   Then from  the Chebysev inequality we have that
\be \label{BK6}
P(\ga\langle  X_t\rangle<X_t<x_0) \ \ge \  1/2-\langle X_t\rangle/x_0 \ \ge \ 1/4 \quad {\rm if \ } \langle X_t\rangle\le x_0/4  \ .
\ee
Hence (\ref{BJ6}), (\ref{BK6}) imply that
\be \label{BL6}
\frac{d}{dt}\log   \ \langle X_t\rangle \ \ge (1-\del)\left[\frac{ \phi(\ga\langle X_t\rangle)}{4\langle X_t\rangle}+1\right]-1 \ \quad {\rm if \ }  \langle X_t\rangle \le 1/K_\del \ , t\ge T_\del \ ,
\ee
provided $K_\del>4/x_0$. Choosing $\del$ now to satisfy $(1-\del)[\ga\phi'(0)/4+1]>1$, we see from (\ref{BL6}) that there exists $T_{1,\del}\ge T_\del$ such that  $\langle X_t\rangle\ge 1/K_\del$ for $t\ge T_{1,\del}$.  We conclude from Lemma 2.3 that $\inf\kappa(\cdot)>0$. 
\end{proof}
\begin{proof}[Proof of Theorem 1.3]
The result follows from Lemma 4.1, Corollary 4.1, Lemma 4.2 and Proposition 6.1.
\end{proof}
We conclude this section by making some observations concerning the conditions (\ref{O1}), (\ref{Y1}) on the functions $\phi(\cdot),   \psi(\cdot)$.  In Lemma 6.1 we saw that (\ref{O1}) implies that the function $z\ra g(z,u)$ is increasing. This fact can also be concluded from (\ref{H6}) and the concavity of the function $z\ra g(z,u)$, which follows from (\ref{Y1}). Therefore the only part of the proof of Theorem 1.3 in which we need to assume (\ref{O1}) is in the proof of Lemma 4.1. We can however replace Lemma 4.1 by the following  proposition in the case when $\lim_{x\ra 0}\phi(x)/x<\infty$, and so dispense entirely with the assumption (\ref{O1}) for the proof of Theorem 1.3.
\begin{proposition}
Assume $\phi(\cdot), \ \psi(\cdot)$ satisfy (\ref{D1}), (\ref{E1}), (\ref{Y1}) and that the initial data for (\ref{A1}), (\ref{B1}) satisfies (\ref{U1}) with $\beta_0=1$. Then if  $\lim_{x\ra 0}\phi(x)/x<\infty$ there is a constant $C$ such that $\kappa(t)\le C$ for $t\ge 0$.  
\end{proposition}
\begin{proof}
It follows from Lemma 6.1 that the function $g(z,u)$ of (\ref{G6}) is negative, increasing and concave for $z\ge f(0)u$.  We first note that the assumption  $\lim_{x\ra 0}\phi(x)/x<\infty$ implies that the function $z\ra g(z,u)$ is $C^1$ on the closed interval $[f(0)u,\infty)$ since $\lim_{x\ra 0}x\psi'(x)=0$.  Hence we can extend $g(z,u)$ to be a $C^1$ function on $[0,\infty)$ by setting $\pa g(z,u)/\pa z= \pa g(f(0)u,u)/\pa z$ for $0\le z\le f(0)u$.  The extended function $g(\cdot,u)$ is negative, increasing, concave and $g(0,u)=-\al_1u$ for some positive constant $\al_1\ge \al_0$, where $\al_0$ is defined by (\ref{H6}). We now define an extended function $\hat{F}(z,t), \ z\ge 0,$ as the solution to the initial value problem (\ref{O6}) in the domain $\{(z,t): \ z>0,t>0\}$, and it is clear that   the extended function $\hat{F}(\cdot,t)$ is increasing convex and $\pa \hat{F}(z,t)/\pa z\le 1$ for $z\ge 0$.  We further define a function $y(t)$ by $f(y(t)) \ = \ \hat{F}(0,t)$ where  the function $f(\cdot)$ is determined by (\ref{C6}). Observe that in the case of quadratic $\phi(\cdot),\psi(\cdot)$ this function coincides with the function $y(t)$ of (\ref{G5}). Since $z+C_1v(t)\le \hat{F}(z,t)\le z+C_2v(t)$ for $z,t\ge 0$  as in Corollary 6.1, we have that  $\lim_{t\ra\infty} \hat{F}(0,t)=\infty$. Hence there  exists $T_0\ge 0$ such that  $y(t)$ is uniquely defined for $t\ge T_0$ and satisfies 
$0<y(t)<F(0,t)<1$. 

Let $k:[f(0),\infty)\ra [0,1)$ be the inverse function of $f:[0,1)\ra [f(0),\infty)$.  Since $f(\cdot)$ is strictly increasing and convex, it follows that $k(\cdot)$ is strictly increasing and concave. We have now from (\ref{P6})  that
\begin{multline} \label{BM6}
F(x,t)-y(t) \ = \ k(\hat{F}(f(x)u(t),t))- k(\hat{F}(0,t)) \ \ge \\
 k'(\hat{F}(f(x)u(t),t))[\hat{F}(f(x)u(t),t)-\hat{F}(0,t)] \ \ge \  k'(\hat{F}(f(x)u(t),t))f(x)u(t)\pa \hat{F}(0,t)/\pa z \ ,
\end{multline}
where in (\ref{BM6})  we have used the fact that the function $z\ra\hat{F}(z,t)$ is increasing and convex. From (\ref{C6}) we see that the function $k(\cdot)$ is $C^1$ on $[f(0),\infty)$ and  satisfies
$\lim_{y\ra \infty} y^2k'(y)=1$. Using the fact that $\lim_{t\ra\infty} v(t)=\infty$, it follows from corollaries 6.1,6.2 that  there are positive constants $C_1,T_1$ such that 
\be \label{BN6}
k'(\hat{F}(f(x)u(t),t)) \ \ge \ C_1k'(\hat{F}(0,t)) \quad {\rm for \ } 0\le x\le 1/2, \ t\ge T_1 \ .
\ee
Hence Lemma 6.1, (\ref{O6}) and (\ref{BM6}), (\ref{BN6}) imply that there is a positive constant $C_2$ such that
\be \label{BO6}
F(x,t)-y(t) \ \ge  \ C_2 \frac{dy(t)}{dt} \quad {\rm for \ }  0\le x\le 1/2, \ t\ge T_1  \ .
\ee

Since $y(t)<F(0,t)<z(t)$ we can define as in $\S5$ a positive function $\tau(t)$ satisfying $y(t)=z(t-\tau(t))$. Following the argument of Lemma 5.1 we see that (\ref{BO6}) and the conservation law (\ref{B1}) imply that there exists $T_0,\tau_0>0$ such that $\tau(t)\ge \tau_0$ for $t\ge T_0$.  Observe next as in (\ref{BM6}) that we have
\begin{eqnarray} \label{BP6} \\ \nonumber
F(x,t)-F(0,t) \ &\ge& \  k'(\hat{F}(f(x)u(t),t))[f(x)-f(0)]u(t)\pa \hat{F}(f(0)u(t),t)/\pa z  \ , \\
F(0,t)-y(t) \ &\le& \  k'(\hat{F}(0,t))f(0)u(t)\pa \hat{F}(f(0)u(t),t)/\pa z  \  . \nonumber
\end{eqnarray}
Hence there exists positive constants $C_2,T_2$ such that
\be \label{BQ6}
F(x,t)-F(0,t) \ \ge \ C_2x[F(0,t)-y(t)] \quad {\rm for \ }  t\ge T_2, 0\le x\le 1/2.
\ee
Suppose now that $e^tw(F(0,t),0)=e^{\del}$ for some $0<\del<\tau_0/2$.  Then (\ref{C4}) implies that there exists $T_\del>0$ such that $F(0,t)-y(t)\ge\tau_0g(F(0,t))/2$ provided $t\ge T_\del$.   Hence (\ref{C4}) and (\ref{BQ6}) imply that for small $\del$  the integral on the LHS of (\ref{B1}) is strictly less than $1$. We conclude that there exists $T_2,\del_2>0$ such that $e^tw(F(0,t),0)\ge e^{\del_1}$ for $t\ge T_2$.  The result follows from Lemma 2.3.
\end{proof}

\thanks{ {\bf Acknowledgement:} This research was partially supported by NSF
under grants DMS-0500608 and DMS-0553487.

\end{document}